\documentclass[a4paper,12pt]{article}

\usepackage[latin1]{inputenc}
\usepackage{natbib}
\usepackage{amsmath}
\usepackage{amsthm}
\usepackage{amssymb} 
\usepackage{stmaryrd} 
\usepackage{ctable}
\usepackage{MnSymbol}
\usepackage{natbib}
\usepackage{mathtools}
\usepackage{endnotes} 
\usepackage{setspace}

\newcommand{\rea}{\mathbb{R}}
\newcommand{\raz}{\mathbb{Q}}
\newcommand{\nat}{\mathbb{N}}

\newcommand{\lf}{\underline{x}}
\newcommand{\rg}{\overline{x}}
\newcommand{\x}{\textbf{\textit{x}}}
\newcommand{\y}{\textbf{\textit{y}}}
\newcommand{\z}{\textbf{\textit{z}}}

\DeclarePairedDelimiter{\abs}{\lvert}{\rvert}

\newtheorem{Def}{Definition}
\newtheorem{Theo}{Theorem}
\newtheorem{cor}{Corollary}
\newtheorem{rem}{Remark}
\newtheorem{pro}{Proposition}
\newtheorem{lem}{Lemma}

\begin{document}

\title {Robust Integrals} 

\author{Salvatore Greco and Fabio Rindone \\ 
\small{\textit{Department of Economics and Business,}}\\
\small{\textit{University of Catania, Corso Italia 55, 95129 Catania, Italy}}\\
\small{E-mail: salgreco@unict.it, frindone@unict.it}
}
\date{}
\maketitle

\begin{abstract}
In decision analysis and especially in multiple criteria decision analysis, several non additive integrals have been introduced in the last years. 
Among them, we remember the Choquet integral, the Shilkret integral and the Sugeno integral. 
In the context of multiple criteria decision analysis, these integrals are used to aggregate the evaluations of possible choice alternatives, with respect to several criteria, into a single overall evaluation. 
These integrals request the starting evaluations to be expressed in terms of \textit{exact-evaluations}.
In this paper we present the \textit{robust} Choquet, Shilkret and Sugeno integrals, computed with respect to an \textit{interval capacity}.
These are quite natural generalizations of the Choquet, Shilkret and Sugeno integrals, useful to aggregate \textit{interval-evaluations} of choice alternatives into a single overall evaluation. 
We show that, when the interval-evaluations collapse into exact-evaluations, our definitions of robust integrals collapse into the previous definitions.
We also provide an axiomatic characterization of the robust Choquet integral.
\end{abstract}


\section{Introduction}
In many decision problems a set of alternatives is evaluated with respect to a set of points of view, called criteria. 
For example, in evaluating a car one can consider criteria such as maximum speed, price, acceleration, fuel consumption.
In evaluating a set of students one can consider as criteria the notes in examinations with respect to different subjects such as Mathematics, Physics, Literature and so on. 
In general, evaluations of an alternative with respect to different criteria can be conflicting with respect to preferences. For example, very often when a car has a good maximum speed, it has also a high price and a high fuel consumption, or if a student is very good in Mathematics, may be not so good in Literature. 
Thus, in order to express a decision such as a choice from a given set of cars or a ranking of a set of students,
it is necessary to aggregate the evaluations on considered criteria, taking into account the possible interactions. 
This is the domain of multiple criteria decision analysis and in this context several methodologies have been proposed (for a collection of extensive state-of-art surveys see \cite{figueira2005multiple}).
Suppose to have n criteria $N=\{1,\ldots,n\}$ and that on each of them the evaluation of a given alternative $\x$ is expressed by a single number (on the same scale).  
Thus, such an alternative can be identified with a score vector 
$\x=(x_1,\ldots,x_n)$, where $x_i\in \rea$ represents the evaluation of $\x$ with respect to the $i^{th}$ criterion. 
If the criteria are independent, a natural way to aggregate the $x_i$ is using the weighted arithmetic means 
$E_w(\x)=\sum_1^nw_ix_i$ with $\sum_1^nw_i=1$ and $w_i\ge 0$.
When the criteria are interacting the weighted arithmetic means must be substituted with non additive operators.
In the last years, several non additive integrals have been developed in order to obtain an aggregated evaluation of $\x$, say $E(\x)$ (for a comprehensive survey see \cite{grabisch2005fuzzy}).
Among them we remember the Choquet integral \cite{choquet1953theory}, the Shilkret integral \cite{shilkret1971maxitive} and the Sugeno integral \cite{sugeno1974theory}.
All these integrals are computed with respect to a capacity \cite{choquet1953theory} or fuzzy measure \cite{sugeno1974theory} allowing the  importance of a set of criteria to be not necessarily the sum of the importance of each criterion in the set. 
It can be smaller or greater, due respectively to redundancy or synergy among criteria. 
These integrals can be used if the starting evaluations are exactly expressed (on a numerical or ordinal scale).
However, in the real life it is very simple to image situations where we have only partial informations about the possible evaluations on each criterion. 
Specifically, on this paper we face the case of \textit{interval-evaluations}.
For example, suppose a situation where, considering only two criteria, 
an alternative $\x$ is evaluated between 5 and 10 on on the first criterion and between 7 and 20 on the second.
Again $\x$ can be represented as a score vector  
$\x=\left(\left[5,10\right],\left[7,20\right]\right)$.
Using a generic aggregation operator $E$, it seems natural to aggregate separately the $\x$ "pessimistic" evaluations 
$\x_*=\left(5,7\right)$ and the 
"optimistic" ones, 
$\x^*=\left(10,20\right)$, in order to obtain an interval 
$[E(\x_*),E(\x^*)]$
containing the global evaluation of $\x$.
If we wish to obtain such a global evaluation, we should furthermore aggregate 
$E(\x_*)$ and $E(\x^*)$ 
into a single number.
Thus, the aggregation of interval evaluations into an exact evaluation should necessarily request two steps.
In this paper we aim to synthesize these two processes into one single aggregating process. 
To this purpose we provide a quite natural generalization of the classical Choquet, Shilkret and Sugeno integrals, which we call the \textit{robust} Choquet, Shilkret and Sugeno integrals computed wit respect to an \textit{interval capacity}.
Roughly speaking, our integrals are special case of integrals of set valued functions \cite{aumann1965integrals}.
Another question we face is that of order on the set of intervals. 
It is well known that the philosophy of the Choquet integral applied to a given alternative is based on the ranking of the alternative evaluations on the various criteria.
Being these evaluations single numbers, their ranking agrees with the natural order of $\rea$. 
In the case of interval evaluations, we have not a ``natural order'' to be preserved, like in $\rea$.
On the other hand we want that an evaluation on the range $[5,10]$ is considered better than an evaluation on the range $[1,4]$ and, then, some assumption about a primitive ordering on intervals must be done.
One choice could be to assume the lexicographic order: $[a,b]\prec [a',b']$ iff $a<a'$ or $a=a'$ and $b<b'$.
The lexicographic order has the advantage to be a complete order, but it leads to the conclusion that $[2.99,100]\prec [3,4]$, which we do not consider a suitable conclusion in the case of interval evaluations.
Instead, through this paper we shall assume as desirable order on intervals to be preserved that defined by considering an evaluation on the range 
$[a,b]$ 
better or equal than an evaluation on the range $[a',b']$ iff 
$a\ge a'$ and $b\ge b'$.\\
Finally, we wish to remember as in contrast to the fact that in real life decisions we often face imprecise evaluations, 
in multiple criteria decision analysis little has been developed in order to provide appropriate tools to aggregate such evaluations.
In the best of our knowledge this question has been only partially treated in \cite{jang2004interval, bustince2012multicriteria}.
\\
The paper is organized as follows.
Section 2 contains the basic concepts.
In section 3 we give the definition of Robust Choquet Integral (RCI) computed with respect to an interval capacity.
In section 4 we give an illustrative application of the RCI, 
while in section 5 we provide a full axiomatic characterization of this integral. 
In section 6 we explore the possibility of rewriting the RCI by means of its M{\"o}bius inverse.   
In section 7 we give the definitions of robust Sugeno and Shilkret integrals and in section 8 we apply our generalization to other fuzzy integrals, among them to the concave integral of Lehrer \cite{lehrer2008concave}.
In section 9 we extend our discussion to the case of m-point intervals \cite{ozturk2011representing}.
In section 10 we present our conclusions.

\section{Basic concepts}
Let us consider a set of \textit{alternative}
$A=\left\{\x,\textit{\textbf{y}},\textit{\textbf{z}},\ldots \right\}$
to be evaluated with respect to a set of criteria  
$N=\left\{1,\cdots,n\right\}$.
Suppose that for every $\x\in A$, we have, on each criterion, a numerical imprecise evaluation. 
Specifically, suppose that for each $i\in N$ we know a range $\left[\lf _i,\rg _i\right]$ containing the exact evaluation of $\x$ with respect to $i$. 
Thus, being 
$\mathcal I=\left\{[a,b]\ |\  a,b\in\rea,\  a\le b\right\}$ the set of bounded and closed intervals of $\rea$,  
any alternative $\x$ can be identified with a score vector  
\begin{equation}
\x=\left(\left[\lf _1,\rg _1\right],\ldots,\ \left[\lf _i,\rg _i\right],\ldots,\left[\lf _n,\rg _n\right]\right)\in\mathcal I^n
\label{eq:vector}
\end{equation}
whose $i^{th}$ component, 
$[\lf_i,\rg_i]$, 
is the interval containing the evaluation of $\x$ with respect to the $i^{th}$ criterion. 
Vectors of $\rea^n$ are considered elements of $\mathcal I^n$ by identifying each $x\in\rea$ with the degenerate interval (or singleton) $[x,x]=\{x\}$.
Thus, with a slight abuse of notation, we write $[x,x]=x$.
We associate to every 
$\x=\left(\left[\lf _1,\rg _1\right],\ldots,\left[\lf _n,\rg _n\right]\right)\in \mathcal I^n$ the vector 
$\underline{\x}=\left(\lf_1,\ldots,\lf_n\right)$ 
of all the worst (or pessimistic) evaluations of $\x$ on each criterion and 
the vector 
$\overline{\x}=\left(\rg_1,\ldots,\rg_n\right)$ 
of all the best (or optimistic) evaluations of $\x$ on each criterion. 
Trough the paper, the elements of $\mathcal I^n$ will be, indifferently, called alternatives or vectors. 
\\
Let us consider the set  $\mathcal Q=\left\{(A,B)\ |\ A\subseteq B\subseteq N \right\}$ of all pairs of subsets of $N$ in which the first component is included in the second. 
With a slight abuse of notation we extend to $\mathcal Q$ the set relations of inclusion, union and intersection by defining 
for all $(A,B), (C,D) \in \mathcal Q$:
\[
(A,B)\subseteq (C,D)\ \textnormal{if and only if}\ A \subseteq C \ \textnormal{and}\ B\subseteq D,
\]
\[
\left(A,B\right)\cup\left(C,D\right)=\left(A\cup C, B\cup D\right),
\]
\[
\left(A,B\right)\cap\left(C,D\right)=\left(A\cap C, B\cap D\right).
\]
Regarding the algebraic structure of $\mathcal Q$, we can observe that with respect to the relation $\subseteq$, $\mathcal Q$ is a lattice, i.e. a partial ordered set in which every two elements have a unique supremum and a unique infimum.
Those are given, for all $(A,B), (C,D) \in \mathcal Q$, respectively, by 
$$
\sup\left\{\left(A,B\right),\left(C,D\right)\right\}=\left(A,B\right)\cup\left(C,D\right),
$$ 
$$
\inf\left\{\left(A,B\right),\left(C,D\right)\right\}=\left(A,B\right)\cap\left(C,D\right).
$$
Moreover the lattice $\left(\mathcal Q, \subseteq\right)$ is also distributive. 
Indeed, due to the distributive property of set union over intersection (and vice versa) we have that 
\[
\left(A,B\right)\cup \left[\left(C,D\right)\cap\left(E,F\right)\right]=
\left[\left(A,B\right)\cup\left(C,D\right)\right]\cap\left[\left(A,B\right)\cup\left(E,F\right)\right],
\]
\[
\left(A,B\right)\cap \left[\left(C,D\right)\cup\left(E,F\right)\right]=
\left[\left(A,B\right)\cap\left(C,D\right)\right]\cup\left[\left(A,B\right)\cap\left(E,F\right)\right].
\]
Regarding the significance of $\mathcal Q$ in this work, let us consider  
$\x=\left(\left[\lf _1,\rg _1\right],\ldots,\left[\lf _n,\rg _n\right]\right)\in\mathcal I^n$
and a fixed evaluation level $t\in\rea$. 
In the pair 
\[
(A_t,B_t)=(\{i \in N\  |\  \lf_i \geq t \},\{i \in N\ |\ \rg_i\geq t \}),
\]
$A_t$ aggregates the criteria whose pessimistic evaluation of $\x$ is at least $t$, while $B_t$ aggregates the criteria whose optimistic evaluation of $\x$ is at least $t$. 
Clearly, $A_t\subseteq B_t\subseteq N$ and thus $(A_t,B_t)\in \mathcal Q$ for all $t\in \rea$. 
We aim to define a tool allowing for the assignment of a ``weight'' to such elements of $\mathcal Q$.

\section{The robust Choquet integral}

\begin{Def} 
A function $\mu_r: \mathcal Q \rightarrow [0,1]$ is an \textit{interval-capacity} on $\mathcal Q$  if
\begin{itemize}
	\item $\mu_r(\emptyset,  \emptyset)=0$ and $\mu_r(N,N)=1$;
	\item $\mu_r(A,B) \leq \mu_r(C,D)$ for all 
		$(A,B), (C,D) \in \mathcal Q$ such that $(A, B) \subseteq (C,D).$
\end{itemize}
\label{def:interval capacity}
\end{Def}

\begin{Def}
The Robust Choquet Integral (RCI) of 
$\x=\left(\left[\lf _1,\rg _1\right],\ldots,\left[\lf _n,\rg _n\right]\right)\in \mathcal I^n$ with respect to an interval capacity $\mu_r:\mathcal Q\rightarrow [0,1]$ is given by:
\begin{equation}
Ch_r\left(\x,\mu_r\right)=: 
\int_{\min\left\{\lf_1,\ldots,\lf_n\right\}}^{\max\left\{\rg_1,\ldots,\rg_n\right\}} \mu_r(\{i \in N\  |\  \lf_i \geq t \},\{i \in N\ |\ \rg_i\geq t \})dt\ +\ \min\left\{\lf_1,\ldots,\lf_n\right\}.
\label{eq:RCI}
\end{equation}
\end{Def}
\noindent
Note that, being in the \eqref{eq:RCI} the integrand bounded and not increasing, the integral is the standard Riemann integral.\\
An alternative formulation of the RCI implies some additional notations. 
We identify every vector  
$\x=\left(\left[\lf _1,\rg _1\right],\ldots,\left[\lf _n,\rg _n\right]\right)\in \mathcal I^n$ 
with the vector 
$\x^*=(x_1,\ldots,x_{2n})\in\rea^{2n}$ 
defined by setting for all $i=1,\ldots, 2n$:  

\begin{equation}
x_i=\left\{
\begin{array}{ll}
	\lf_i & i\leq n 
	\\
	\rg_{i-n} & i> n.
\end{array}
\right.
\label{eq:ident}
\end{equation}
This corresponds to identify $\x\in \mathcal I^n$ with 
$\x^*=(x_1,\ldots,x_{2n})=\left(\lf _1,\ldots, \lf_n,\rg_1,\ldots,\rg _n \right)\in \rea^{2n}$. 
Now, let 
$(\cdot):\left\{1,\ldots,2n\right\}\rightarrow \left\{1,\ldots,2n\right\}$
be a permutation of indices such that  
$x_{(1)}\leq x_{(2)}\leq\ldots\leq x_{(2n)}$
and for all $i=1,\ldots,2n$ let us define 
$A_{(i)}=\left\{j\in N\ |\ \lf_j\geq x_{(i)}\right\}$ and 
$B_{(i)}=\left\{j\in N\ |\ \rg_j\geq x_{(i)}\right\}$.
Thus, two alternative formulations of \eqref{eq:RCI} are:
\begin{equation}
Ch_r\left(\x,\mu_r\right) =\sum_{i=2}^{2n}{\left(x_{(i)}-x_{(i-1)}\right)\mu_r\left(A_{(i)},B_{(i)}\right)}\ + x_{(1)}
\label{eq:RCI1}
\end{equation}
and
\begin{equation}
Ch_r\left(\x,\mu_r\right)=
\sum_{i=1}^{2n}{x_{(i)}\left[\mu_r\left(A_{(i)},B_{(i)}\right)-\mu_r\left(A_{(i+1)},B_{(i+1)}\right)\right]}.  
\label{eq:RCI2}
\end{equation}

\subsection{Interpretation}
The indicator function of a set $A\subseteq N$ is the function 
$1_A:N\rightarrow \{0,1\}$ 
which attains 1 on $A$ and 0 elsewhere. 
Such a function can be identified with the vector 
$\textbf{1}_A\in\rea^n$ 
whose $i^{th}$ component equals 1 if $i\in A$ and equals 0 if $i\notin A$.
For all $\left(A,B\right)\in \mathcal Q$ the generalized indicator function $1_{\left(A,B\right)}:N\rightarrow\{0,1,[0,1]\}$ is defined by setting for all $i\in N$ 
\begin{equation}
1_{\left(A,B\right)}(i)=
\left\{
\begin{array}{lll}
[1,1]=1 &  i\in A \\
\left[0,1\right] &  i\in B\setminus A \\
\left[0,0\right]=0 &  i\in N\setminus B.
\end{array}\right. 
\label{eq:indicator}
\end{equation} 

\noindent The \eqref{eq:indicator} can be thought as the function indicating ``$A$ for sure and, eventually, $B\setminus A$.''
Clearly, if $A=B$, $1_{\left(A,A\right)}=1_A$. 
The function $1_{\left(A,B\right)}$ can be identified with the vector $\textbf{1}_{\left(A,B\right)}\in\mathcal I^n$ 
whose $i^{th}$ component equals $[1,1]=1$ if $i\in A$, 
equals $[0,1]$ if $i\in B\setminus A$ 
and equals 0 if $i\notin B$.\\
It follows by the definition of RCI that for any interval capacity $\mu_r$: 
\begin{equation}
Ch_r(\textbf{1}_{\left(A,B\right)},\mu_r)=\mu_r\left(A,B\right).
\label{eq:interpretation}
\end{equation}
This relation offers an appropriate definition of the weights $\mu_r\left(A,B\right)$. 
Indeed, provided that the partial score $[\lf_i,\rg_i]$ are contained in $[0,1]$, the \eqref{eq:interpretation} suggests that the weight of importance of any couple $(A,B)\in \mathcal Q$ is defined as the global evaluation of the alternative that 
\begin{itemize}
	\item completely satisfies the criteria from $A$,
	\item have an unknown degree of satisfaction (on the scale $[0,1]$) about the criteria from $B\setminus A$,
	\item totally fails to satisfy the criteria from $N\setminus B$.
\end{itemize}

\subsection{Relation with the Choquet Integral}
A capacity \cite{choquet1953theory} or fuzzy measure \cite{sugeno1974theory} on $N$ is a non decreasing set function $\nu:2^N\rightarrow [0,1]$ such that $\nu(\emptyset)=0$ and $\nu(N)=1$.
\begin{Def}
The Choquet integral \cite{choquet1953theory} of a vector $\textbf{\textit{x}}=\left(x_1, \ldots, x_n\right)\in  \left[0,+\infty\right.\left[\right.^n$ with respect to the capacity $\nu$ is given by
\begin{equation}
Ch(\textbf{\textit{x}},\mu)=\int_{0}^{\infty} \mu\left(\{i\in N: x_i \ge t\}\right) dt.
\label{choquet}
\end{equation}
\end{Def}

\noindent Schmeidler \cite{schmeidler1986integral} extended the above definition to negative values too.  

\begin{Def}
The Choquet integral \cite{schmeidler1986integral} of a vector 
$\textbf{\textit{x}}=\left(x_1, \ldots, x_n\right)\in \rea^n $ 
with respect to the capacity $\nu$ is given by
\begin{equation}
Ch(\textbf{\textit{x}},\nu)=\int_{\min_i x_i}^{\max_i x_i} \nu\left(\{i\in N: x_i \ge t\}\right) dt +\min_i x_i.
\label{choquetschmeidler}
\end{equation}
\end{Def}

\noindent Alternatively \eqref{choquetschmeidler} can be written as 
\begin{equation}
Ch(\textbf{\textit{x}},\nu)=\sum_{i=2}^n\left(x_{(i)}-x_{(i-1)}\right)\cdot\nu\left(\{j\in N: x_j \ge x_{(i)}\}\right)\ +\ x_{(1)},
\label{choquet1}
\end{equation}
being $():N\rightarrow N$ any permutation of indexes such that $x_{(1)}\le\ldots\le x_{(n)}$.\\
Now, suppose to have 
$\x=\left(\left[\lf _1,\rg _1\right],\ldots,\left[\lf _n,\rg _n\right]\right)\in \mathcal I^n$ 
such that $\lf_i=\rg_i$ for all $i\in N$, thus $\x\in\rea^n$.
Let be given an interval capacity $\mu_r: \mathcal Q \rightarrow [0,1]$.
It is straightforward to note that
$\nu(A)=\mu_r(A,A):2^N\rightarrow [0,1]$ defines a capacity.
In this case the RCI of $\x$ with respect to $\mu_r$ collapses on the Choquet integral of $\x$ with respect to $\nu$, i.e. $Ch_r(\x,\mu_r) =Ch(\x,\nu)$.\\
Moreover, the RCI is a monotonic functional (see section \ref{sec:characterization}) and then for all $\x\in\mathcal I^n$,
\begin{equation}
Ch_r(\underline\x,\mu_r)=Ch(\underline\x,\nu)\leq Ch_r(\x,\mu_r)\leq Ch(\overline{\x},\nu)=Ch_r(\overline\x,\mu_r).
\label{eq:boh}
\end{equation}
If $\mu_r(\emptyset, N)\notin\{0,1\}$, other two capacities can be elicited from $\mu_r$ by setting for all $A\subseteq N$
\[
\underline\nu(A)=\frac{\mu_r(A,N)-\mu_r(\emptyset,N)}{1-\mu_r(\emptyset,N)} \quad \textnormal{and}\quad 
\overline\nu(A)=\frac{\mu_r(\emptyset,A)}{\mu_r(\emptyset,N)}.  
\]
These two capacities naturally arise in the proof of proposition \ref{pro:separable}.\\
Now let us examine the relation between the Choquet integral and the RCI in the other verse.
Starting from two capacities, 
$\underline\nu:2^N\rightarrow [0,1]$ and 
$\overline\nu:2^N\rightarrow [0,1]$, 
we can define an interval capacity for every $\alpha\in (0,1)$ by means of  
\begin{equation}
\mu_r(A,B)=\alpha\underline{\nu}(A) + (1-\alpha)\overline{\nu}(B),\ \textnormal{for all}\ (A,B)\in\mathcal Q.
\label{eq:separable}
\end{equation}
\begin{Def}
An interval capacity $\mu_r(A,B): \mathcal Q \rightarrow [0,1]$ is said separable if there exist an $\alpha\in (0,1)$ and two capacities, 
$\underline\nu:2^N\rightarrow [0,1]$ and 
$\overline\nu:2^N\rightarrow [0,1]$, 
such that the \eqref{eq:separable} holds.
\end{Def}
\begin{pro}
An interval capacity $\mu_r(A,B): \mathcal Q \rightarrow [0,1]$ is separable if and only if for every 
$A,A',B,B'\in2^N$ with 
$A\cup A'\subseteq B\cap B'$ it holds the 
\begin{equation}
\mu_r(A,B)-\mu_r(A',B)=\mu_r(A,B')-\mu_r(A',B').
\label{eq:sep}
\end{equation} 
\label{pro:separable}
\end{pro}
\begin{proof}
Let us note that the \eqref{eq:sep} can be rewritten as 
\begin{equation}
\mu_r(A',B')-\mu_r(A',B)=\mu_r(A,B')-\mu_r(A,B).
\label{eq:sep1}
\end{equation} 
Thus the condition \eqref{eq:sep} means that the difference between two interval capacities is independent from common coalitions of criteria in the first or in the second argument.
The necessary part of the theorem is trivial, let us prove the sufficient part.
Suppose that $\mu_r$ is an interval capacity satisfying the \eqref{eq:sep}. 
Thus if $A'=\emptyset$ and $B'=N$ and if $\mu_r(\emptyset,B)\notin\{0,1\}$ we get:
\[
\mu_r(A,B)=\mu_r(A,N)-\mu_r(\emptyset,N)+\mu_r(\emptyset,B)=
\frac{\mu_r(A,N)-\mu_r(\emptyset,N)}{1-\mu_r(\emptyset,N)}\left(1-\mu_r(\emptyset,N)\right)+
\frac{\mu_r(\emptyset,B)}{\mu_r(\emptyset,N)}\mu_r(\emptyset,N).
\]
In this case $\mu_r$ is separable taking for all $A,B\in 2^N$, 
\[
\alpha=1-\mu_r(\emptyset,B),\quad 
\underline\nu(A)=\frac{\mu_r(A,N)-\mu_r(\emptyset,N)}{1-\mu_r(\emptyset,N)} \quad \textnormal{and}\quad 
\overline\nu(B)=\frac{\mu_r(\emptyset,B)}{\mu_r(\emptyset,N)}.  
\]
If $\mu_r(\emptyset,N)=0$ we take $\alpha=1$ and $\underline\nu(A)=\mu_r(A,N)$.
Finally, if $\mu_r(\emptyset,N)=1$ we take $\alpha=0$ and $\overline\nu(B)=\mu_r(\emptyset,B)$.
\end{proof}

It is easy to  verify that if $\mu_r$ is a separable interval capacity defined according to \eqref{eq:separable},
the RCI of every $\x\in\mathcal I^n$ is the mixture of the two Choquet integrals of $\underline\x,  \overline{\x}\in\rea^n$ computed, respectively, with respect to $\underline{\nu}$ and $\overline{\nu}$:
\begin{equation}
Ch_r(\x,\mu_r)=\alpha Ch(\underline\x, \underline{\nu}) + (1-\alpha)Ch(\overline{\x}, \overline{\nu}).
\label{eq:cho2}
\end{equation}
In the case of a single capacity, $\underline{\nu}=\overline{\nu}=\nu$, one could think to obtain a lower, an intermediate and an upper aggregate evaluation of an alternative $\x\in\mathcal I^n$ by means of 
\begin{equation}
Ch(\underline\x,\nu)\leq \alpha Ch(\underline\x, \nu) + (1-\alpha)Ch(\overline{\x}, \nu)\leq Ch(\overline{\x}, \nu).
\label{eq:cho3}
\end{equation}
The mixture $\alpha Ch(\underline\x, \nu) + (1-\alpha)Ch(\overline{\x}, \nu)$ is the RCI of $\x$ with respect to a separable interval capacity $\mu_r$. 
Clearly, our approach is more general since it does not impose the separability of $\mu_r$.

\section{An illustrative example}

Taking inspiration from an example very well known in the specialized literature \cite{grabisch1996application} let us consider a case of evaluation of students. 
A typical situation, which can arise in the middle of a school year, is that when some teachers, being not sure about the evaluation of a student, express it in terms of an interval. 
Perhaps it is not a great lack of information to know that a student is evaluated in Mathematics between 5 and 6.
But the problems can arise when we must compare several students having imprecise evaluations and, to this scope, we need an aggregated evaluation of each student.\\
We suppose that the students are evaluated on each subject on a 10 point scale. 
Let us suppose that we globally evaluate students with respect to evaluations in Mathematics, Physics and Literature.
Let us consider three students having the evaluations presented in Table \ref{tab:2}. 
As can be seen, some evaluations are imprecise.
Suppose also that the dean of the school ranks the students as follows:
$$S_2\succ S_1 \succ S_3.$$

\begin{table}
\begin{center}
\begin{tabular}{|c c c c|}
\hline
& & &\\
 & 
Mathematics 
& 
	Physics 	
	&
Literature
\\
& & & \\
\hline
& & &\\
$S_1$ &
8 
&
8
&
7\\
& & &\\
& & &\\
$S_2$
& 
$[7,8]$
&
8
&
$[6,8]$
\\
& & &\\
&&&\\

$S_3$ & 9 & 9 & $[5,6]$
\\
& & & \\
\hline
\end{tabular}
\caption{Students' evaluations  }
\label{tab:2}
\end{center}
\end{table}
\noindent The rationale of this ranking is that:
\begin{itemize}
	\item $S_1\succ S_3$ since the better evaluations of $S_3$ in scientific subjects, i.e. Mathematics and Physics are redundant, and the dean retains relevant the better evaluation of $S_1$ in Literature, where $S_3$ risks an insufficiency. 
	In other words, when the scientific evaluation is fairly high, Literature becomes very important;
	\item $S_2\succ S_1$ since the conjoint evaluation in Mathematics and Physics is very similar, also considering the redundancy of the two subjects. 
	However $S_2$ has the same average in Literature and, then, a greater potential;
	\item $S_2\succ S_3$ by transitivity of preferences.
	\end{itemize}
Let us note that, if we consider separately the three averages given by the minimum, central and maximum evaluations of each student for each subject, see Table \ref{tab:3}, we cannot explain the (rational) preferences of the dean.
On the contrary, the evidence of such average evaluations shows how we should consider $S_3$ the best student.
Next we show how the RCI permits to represent the preferences of the dean. 
Let $N=\left\{M,\ Ph,\ L\right\}$ be the set of criteria and let us identify the three students (alternative) $S_1, S_2$ and $S_3$, respectively with the three vectors:
\[
\begin{array}{l}
\x_1=\left([8,8],[8,8],[7,7]\right),\\
\x_2=\left([7,8],[8,8],[6,8]\right),\\
\x_3=\left([9,9],[9,9],[5,6]\right).
\end{array}
\]
The RCI represents the preferences of the dean if there exists an interval capacity $\mu_r$ such that
\[Ch_r(\x_2,\mu_r)>Ch_r(\x_1, \mu_r) >Ch_r(\x_3,\mu_r),\]
that is
\[
6+\mu_r\left(\left\{M,Ph\right\},N\right)+\mu_r\left(\left\{Ph\right\},N\right)>
7+\mu_r\left(\left\{M,Ph\right\},\left\{M,Ph\right\}\right)>
\]
\[
>5+\mu_r\left(\left\{M,Ph\right\},N\right)+3\mu_r\left(\left\{M,Ph\right\},\left\{M,Ph\right\}\right).
\]
Which can be explained, for example, by setting 
\[
\left\{
\begin{array}{l}
	\mu_r\left(\left\{M,Ph\right\},N\right)= 0.9\\
\mu_r\left(\left\{Ph\right\},N\right)=0.7\\
\mu_r\left(\left\{M,Ph\right\},\left\{M,Ph\right\}\right)=0.5.
\end{array}
\right.
\]
\begin{table}
\begin{center}
\begin{tabular}{|c c c c|}
\hline
& & &\\
 & 
minimum
& 
medium 	
	&
maximum
\\
& & & \\
\hline
& & &\\
$S_1$ & 7.67 & 7.67 & 7.67
\\
& & &\\
& & &\\
$S_2$ & 7 & 7.5 & 8
\\
& & &\\
&&&\\

$S_3$ & 7.67 & 7.83 & 8
\\
& & & \\
\hline
\end{tabular}
\caption{Average evaluations}
\label{tab:3}
\end{center}
\end{table}

\section{Axiomatic characterization of the RCI}{\label{sec:characterization}}
Let us recall some well known definitions.
Consider two vectors (alternatives) of $\rea^n$, 
$\x=(x_1,\ldots,x_n)$ and  $\y=(y_1,\ldots,y_n)$. 
We say that $\x$ dominates $\y$ if for all $i\in N$ $x_i\ge y_i$ and in this case we simply write $\x\ge\y$.
We say that $\x$ and $\y$ are comonotone if $(x_i-x_j)(y_i-y_j) \ge 0$ for all $i,j \in N$.
A monotone function $G: \rea^n\rightarrow\rea$ is a function such that 
$G(\x) \ge G(\y)$ whenever $\x\ge \y$. 
In the context of multiple criteria decision analysis, monotone functions are called aggregation functions. 
They are useful tools to aggregate $n$ evaluations of an alternative into a single evaluation.
An aggregation function $G: \rea^n\rightarrow\rea$ is:
\begin{itemize}
	\item idempotent, if for all constant vector $\textbf{a}=(a,\ldots,a)\in\rea^n$, $G(\textbf{a})=a$;
	\item homogeneous, if for all $\x\in\rea^n$ and $c>0$,  $G(c\cdot\x)=c\cdot G(\x)$;
	\item comonotone additive, if for all comonotone $\x, \y\in\rea^n$, $G(\x+\y)=G(\x)+G(\y)$.
	\end{itemize}
In \cite{schmeidler1986integral} it has been showed that the Choquet integral is an idempotent, homogeneous and comonotone additive aggregation function.
Moreover, these properties are also characterizing the Choquet integral, as showed by the following theorem.
\begin{Theo} \cite{schmeidler1986integral} A monotone function $G: \rea^n \rightarrow\rea$ satisfying $G(\textbf{1}_N)=1$ is  comonotone additive if and only if there exists a  capacity $\nu$ such that, for all $\x\in \rea^n$,
$$G(\textbf{\textit{x}})=Ch(\textbf{\textit{x}},\nu).$$
\end{Theo}
\noindent Note that homogeneity is not among the hypotheses of the theorem since it can be elicited from monotonicity and comonotone additivity. Moreover from homogeneity and the condition $G(\textbf{1}_N)=1$ we also elicit idempotency of $G$.\\
Now we turn our attention to the RCI.
As we shall soon see, the RCI with respect to an interval capacity $\mu_r$,  
can be considered a generalized aggregation function. 
This means a monotone function, $Ch_r(.,\mu_r):\mathcal I^n\rightarrow \rea$, transforming vectors of interval evaluations into a single overall numerical evaluation of that alternative.
In order to provide an axiomatic characterization of the RCI we need to extend the notions of monotonicity, idempotency, homogeneity and comonotone additivity for a generic function $G:\mathcal I^n\rightarrow \rea$. 
To this purpose we introduce on $\mathcal I$ and on $\mathcal I^n$, a mixture operation and a preference relation.

\begin{Def}
For every $a\in\rea^+$ and $[x_1,x_2]\in \mathcal I$ we define: 
$a\cdot[x_1,x_2]=[a x_1, a x_2].$
Moreover, for every  
$\x=\left(\left[\lf _1,\rg _1\right],\ldots,\left[\lf _n,\rg _n\right]\right)\in\mathcal I^n$ 
we define  
$a\cdot\x$  
as the element of $I^n$ whose $i^{th}$ component is 
$a\cdot\left[\lf _i,\rg _i\right]$, 
for all $i=1,\ldots,n$.
\end{Def}

\begin{Def}
For every $[x_1,x_2], [y_1,y_2]\in \mathcal I$ we define:
\[
[x_1,x_2] + [y_1,y_2]=[x_1 + y_1,  x_2 +y_2].
\]
Moreover, for every pair of vectors of $\mathcal I^n$,  
$\x=\left(\left[\lf _1,\rg _1\right],\ldots,\left[\lf _n,\rg _n\right]\right)$ 
and 
$\textbf{y}=([\underline{y}_1,\overline{y}_1],\ldots,[\underline{y}_n,\overline{y}_n])$, 
we define  
$\x + \y$ 
as the element of $I^n$ whose $i^{th}$ component is 
$\left[\lf _i,\rg _i\right] + [\underline{y}_i,\overline{y}_i],$ 
for all $i=1,\ldots,n$.
\end{Def}
\noindent Let us note that the two previous definitions can be summarized as follows.
For every $a,b\in\rea^+$ and $[x_1,x_2], [y_1,y_2]\in \mathcal I$ we have the following ``mixture operation'':
\[
a\cdot[x_1,x_2] + b\cdot[y_1,y_2]=[a x_1 + b y_1, a x_2 + b y_2].
\]
Moreover, for every pair of vectors of $\mathcal I^n$, 
$\x=\left(\left[\lf _1,\rg _1\right],\ldots,\left[\lf _n,\rg _n\right]\right)$ 
and 
$\textbf{y}=([\underline{y}_1,\overline{y}_1],\ldots,[\underline{y}_n,\overline{y}_n])$ 
and for all $a,b\in\rea^+$, we have that   
$a\x + b\y$ 
is the element of $I^n$ whose $i^{th}$ component is 
$a\cdot\left[\lf _i,\rg _i\right] + b\cdot[\underline{y}_i,\overline{y}_i],$ 
for all $i=1,\ldots,n$.
\begin{Def}
For all $[\alpha,\beta], [\alpha_1,\beta_1]\in \mathcal I$, we define $[\alpha,\beta]\leq_\mathcal I [\alpha_1,\beta_1]$ whenever $\alpha\leq\alpha_1$ and $\beta\leq\beta_1$.
The symmetric and asymmetric part of $\leq$ on $\mathcal I$ are denoted by $=_\mathcal I$ and $<_\mathcal I$.
Moreover, for every pair of vectors of $\mathcal I^n$, 
$\x=\left(\left[\lf _1,\rg _1\right],\ldots,\left[\lf _n,\rg _n\right]\right)$ 
and 
$\textbf{y}=([\underline{y}_1,\overline{y}_1],\ldots,[\underline{y}_n,\overline{y}_n])$ 
we write $\x\le_\mathcal I\y$ whenever 
$\left[\lf _i,\rg _i\right]\le_\mathcal I [\underline{y}_i,\overline{y}_i]$ 
for all $i\in N$.
\end{Def}  
For the sake of simplicity in the remaining part of the paper the relations $\leq_\mathcal I,\ =_\mathcal I$ and $<_\mathcal I$ shall be simply denoted by $\leq,\ =$ and $<$.
\begin{rem}
Alternatively, for all $\x,\y\in\mathcal I^n$ we can say that $\x\leq \textbf{y}$ iff 
$\underline{\x}\leq \underline{\textbf{y}}$ 
and 
$\overline{\x}\leq \overline{\textbf{y}}$.
\end{rem}

Let us note that $(\mathcal I,\leq)$ is a partial ordered set, i.e. $\leq$ is reflexive, antisymmetric and transitive.
However, this relation is not complete, e.g. we are not able to establish the preference between $[2,5]$ and $[3,4]$.
Then, generally, the evaluations of an alternative on the various criteria, cannot be ranked.\\
The notion of comonotonicity can be easily extended to elements of $\mathcal I^n$ identifying every vector  
$\x=\left(\left[\lf _1,\rg _1\right],\ldots,\left[\lf _n,\rg _n\right]\right)\in \mathcal I^n$ 
with the vector 
$\x^*=(x_1,\ldots,x_{2n})=\left(\lf _1,\ldots, \lf_n,\rg_1,\ldots,\rg _n \right)\in \rea^{2n}$, 
according to \eqref{eq:ident}

\begin{Def}
The two vectors of $\mathcal I^n$, 
$\x=\left(\left[\lf _1,\rg _1\right],\ldots,\left[\lf _n,\rg _n\right]\right)$ 
and 
$\textbf{y}=([\underline{y}_1,\overline{y}_1],\ldots,[\underline{y}_n,\overline{y}_n])$ 
are comonotone (or comonotonic) if they are, in $\rea^{2n}$, the two vectors   
$\x^*=(\underline{x}_1,\ldots,\underline{x}_n,\ldots,\overline{x}_1,\ldots,\overline{x}_n)$ 
and 
$\textbf{y}^*=(\underline{y}_i,\ldots,\underline{y}_n,\ldots,\overline{y}_1,\ldots,\overline{y}_n)$. 
\label{co-monotone}
\end{Def}

\noindent 
Clearly a constant vector $\textbf{\textit{k}}=\left(k,k,\ldots,k\right)\in\rea^n$ with $k\in\rea$, is comonotone with every  $\x\in \mathcal I^n$.
Suppose that  
$\x=\left(\left[\lf _1,\rg _1\right],\ldots,\left[\lf _n,\rg _n\right]\right)$ 
and 
$\textbf{y}=([\underline{y}_1,\overline{y}_1],\ldots,[\underline{y}_n,\overline{y}_n])$ 
are two comonotone vectors of $\mathcal I^n$
and consider the correspondent vectors of $\rea^{2n}$, 
$\x^*=(x_1,\ldots,x_{2n})=\left(\lf _1,\ldots, \lf_n,\rg_1,\ldots,\rg _n \right)$ 
and 
$\y^*=(y_1,\ldots,y_{2n})=(\underline{y}_1,\ldots,\underline{y}_n,\ldots,\overline{y}_1,\ldots,\overline{y}_n)$. 
Schmeidler \cite{schmeidler1986integral} showed that there exists a permutation of indexes 
$(\cdot):\left\{1,\ldots,2n\right\}\rightarrow \left\{1,\ldots,2n\right\}$
such that 
$x_{(1)}\leq x_{(2)}\leq\ldots\leq x_{(2n)}$
and 
$y_{(1)}\leq y_{(2)}\leq\ldots\leq y_{(2n)}$. 

\begin{rem}
If $\x$ and $\textbf{y}$ are comonotone, then both 
$\underline{\x}$ and $\underline{\textbf{y}}$ 
are comonotone
as well as  
$\overline{\x}$ and $\overline{\textbf{y}}$.
The reverse is generally false.
For example, if $N=\left\{1,2\right\}$, 
$\x=\left(\left[1,3\right],\left[2,4\right]\right)$ 
and
$\textbf{y}=\left(\left[1,3\right],\left[4,5\right]\right)$ 
are non comonotone, although 
$\underline{\x}$ 
is comonotone with 
$\underline{\textbf{y}}$ 
and 
$\overline{\x}$
is comonotone with 
$\overline{\textbf{y}}$.
\end{rem} 
Let us note that for all $(A,B),(A',B')\in \mathcal Q$, the relation $(A,B)\subseteq (A',B')$, ensures that $\textbf{1}_{\left(A,B\right)}$ and $\textbf{1}_{\left(A',B'\right)}$ are comonotone. 
Note that their sum $\textbf{1}_{\left(A,B\right)} + \textbf{1}_{\left(A',B'\right)}$ 
is comonotone too with the starting vectors (see tab \ref{tab:1}). 
\begin{table}[ht]
\begin{center}
\resizebox{.5\textwidth}{!}{
\begin{tabular}{|c c c c|}
\hline
& & &\\
 & 
$\textbf{1}_{\left(A,B\right)}$
& 
$\textbf{1}_{\left(A',B'\right)}$
&
$\textbf{1}_{\left(A,B\right)}+\textbf{1}_{\left(A',B'\right)}$
\\
& & & \\
\hline
& & &\\
$A$ &
[1,1]&
[1,1]
&
[2,2]\\
& & &\\
$(A'\cap B)\setminus A$& 
$[0,1]$&
[1,1]
&
$[1,2]$
\\
& & &\\
$B\setminus A'$
& 
$[0,1]$
&
$[0,1]$
&
$[0,2]$
\\
& & & \\
$A'\setminus B$
& 
$[0,0]$
&
$[1,1]$
&
[1,1]\\
&& & \\$B'\setminus (A'\cup B)$
& 
$[0,0]$
&
$[0,1]$
&
$[0,1]$
\\
& & &\\\hline
\end{tabular}
}
\caption{comonotone indicator functions.}
\label{tab:1}
\end{center}
\end{table}

\subsection{Properties of the RCI and characterization Theorem}
Let $\mu_r$ be an interval capacity and let $Ch_r(\cdot, \mu_r)$ be the RCI with respect to $\mu_r$. 
Then $Ch_r(\cdot, \mu_r)$ satisfies the following properties.
\begin{itemize}
	\item [(P1)]\textbf{Idempotency}.
	For all $\textbf{k}=\left(k,k,\ldots,k\right)$ with $k\in\rea$, 
$Ch_r(\textbf{k}, \mu_r)=k$.
	\item [(P2)]\textbf{Positive homogeneity}.
	For all $a>0$ and $\x\in \mathcal I^n$, $Ch_r(a\cdot \x,\mu_r)=a\cdot Ch_r(\x,\mu_r).$
	\item [(P3)]\textbf{Monotonicity}. 
	For all $\x,\y \in \mathcal I^n$ with $\x\leq \textbf{y}$, 
	$Ch_r\left(\x,\mu_r\right)\leq Ch_r\left(\y,\mu_r\right)$.
  \item [(P4)] \textbf{Comonotone additivity}.
For all comonotone $\x,\y\in\mathcal I^n$, $Ch_r\left(\x+\y,\mu_r\right)=
Ch_r\left(\x,\mu_r\right)+ 
Ch_r\left(\y,\mu_r\right).$
\end{itemize}
\begin{proof}
(P1) follows trivially by definition of RCI. Let us prove (P2). Fixed $a>0$ and $\x\in\mathcal I^n$, by definition 
\[Ch_r\left(a\cdot \x,\mu_r\right)= 
\int_{\min\left\{a\lf_1,\ldots,a\lf_n\right\}}^{\max\left\{a\rg_1,\ldots,a\rg_n\right\}}\mu_r(\{i \in N\  |\  a\lf_i \geq t \},\{i \in N\ |\ a\rg_i\geq t \})dt\ +\ \min\left\{a\lf_1,a\lf_2,\ldots,a\lf_n\right\}=\]
\[
=a\cdot\int_{a\cdot \min\left\{\lf_1,\ldots,\lf_n\right\}}^{a\cdot \max\left\{\rg_1,\ldots,\rg_n\right\}}\mu_r(\{i \in N\  |\  \lf_i \geq \frac{t}{a} \},\{i \in N\ |\ \rg_i\geq \frac{t}{a} \})d(t/a)\ +\ a\cdot \min\left\{\lf_1,\lf_2,\ldots,\lf_n\right\}=
a\cdot Ch_r\left(\x,\mu_r\right).
\]
In the last passage we change the variable in the integral from $y=t/a$ to $z=y\cdot a$.
\\
To prove (P3) let us note that for all $t\in \rea$ and for all 
$\x,\y \in \mathcal I^n$ with $\x\leq \textbf{y}$, we get that  
$\left\{i\in N:\ \lf_i\geq t\right\}\subseteq\{i\in N:\ \underline{y}_i\geq t\}$ and
$\left\{i\in N:\ \rg_i\geq t\right\}\subseteq\left\{i\in N:\ \overline{y}_i\geq t\right\}$. 
We conclude that the RCI is a monotonic function by definition and invoking the monotonicity of $\mu_r$ and of the Riemann integral.\\
To prove (P4), suppose that 
$\x=\left(\left[\lf _1,\rg _1\right],\ldots,\left[\lf _n,\rg _n\right]\right)$ 
and 
$\textbf{y}=([\underline{y}_1,\overline{y}_1],\ldots,[\underline{y}_n,\overline{y}_n])$ 
are two comonotone vectors of $\mathcal I^n$
and consider the correspondent vectors of $\rea^{2n}$, 
$\x^*=(x_1,\ldots,x_{2n})$ 
and 
$\y^*=(y_1,\ldots,y_{2n})$, defined according to \eqref{eq:ident}.
Thus, there exists a permutation of indexes 
$(\cdot):\left\{1,\ldots,2n\right\}\rightarrow \left\{1,\ldots,2n\right\}$
such that 
$x_{(1)}\leq\ldots\leq x_{(2n)}$
and 
$y_{(1)}\leq\ldots\leq y_{(2n)}$ 
or equivalently (being $\x^*$ and $\y^*$ comonotone), 
$x_{(1)}+y_{(1)}\leq\ldots\leq x_{(2n)}+y_{(2n)}.$ 
By setting for all $i=1,\ldots,2n$
\begin{eqnarray}
A_{(i)}=\big\{j\in N\ |\ \lf_j\geq x_{(i)}\big\}\  \bigcap \  \big\{j\in N\  |\ \underline{y}_j\geq y_{(i)}\big\},\nonumber\\
B_{(i)}=\big\{j\in N\ |\ \rg_j\geq x_{(i)}\big\}\  \bigcap \ \big\{j\in N\ |\ \overline{y}_j\geq y_{(i)}\big\},
\label{eq:setting}
\end{eqnarray}
we have that 

\begin{eqnarray}
Ch_r\left(\x,\mu_r\right) =\sum_{i=2}^{2n}{\left(x_{(i)}-x_{(i-1)}\right)\mu_r\left(A_{(i)},B_{(i)}\right)}\ + x_{(1)},
\nonumber\\
Ch_r\left(\y,\mu_r\right) =\sum_{i=2}^{2n}{\left(y_{(i)}-y_{(i-1)}\right)\mu_r\left(A_{(i)},B_{(i)}\right)}\ + y_{(1)},
\label{eq:com}
\end{eqnarray}
and also 
\begin{equation}
Ch_r\left(\x+\y,\mu_r\right) =\sum_{i=2}^{2n}{\left(x_{(i)}+ y_{(i)}-x_{(i-1)}-y_{(i-1)}\right)
\mu_r\left(A_{(i)},B_{(i)}\right)}\ + x_{(1)}+ y_{(1)}.
\label{eq:com1}
\end{equation}
From \eqref{eq:com} and \eqref{eq:com1}, comonotone additivity is obtained.
\end{proof}

\begin{rem}
Since the RCI is additive on comonotone vectors and being a constant vector comonotone with all vectors, it follows that the RCI is translational invariant. 
This means that for all $\x\in\mathcal I^n$ and for all $\textbf{k}=(k,\ldots,k)\in\rea^n$, 
$Ch_r(\x+\textbf{\textit{k}},\mu_r)=k + Ch_r(\x,\mu_r)$.  
\end{rem}

\noindent The next theorem establishes that, the above properties are characterizing for the RCI.

\begin{Theo}\label{Theo:maintheo}
Let $G :\mathcal I^n\rightarrow \rea$ be a function satisfying 
\begin{itemize}
	\item $G(\textbf{1}_{(N,N)})=1$, 
	\item (P3) Monotonicity,
	\item (P4) Comonotone additivity.
\end{itemize}
Thus, by assuming  $\mu_r(A,B)=G\left(\textbf{1}_{(A,B)}\right)$ for all $(A,B) \in \mathcal Q$, 
$$\ \ G(\x,\mu_r)=Ch_r(\x,\mu_r), \quad\textnormal{for all}\quad \x\in \mathcal I^n.$$
\end{Theo}
\begin{proof}
First let us note that the properties (P1) and (P2), are not among the hypotheses of Theorem \ref{Theo:maintheo} since they are implied by comonotone additivity (P4), 
monotonicity (P3) and the condition $G(\textbf{1}_{(N,N)})=1$.
Altought the proof of this claim is similar to that in \cite{schmeidler1986integral}, for the sake of clarity, we recall it here. 
Regarding the homogeneity, if $n\in\nat$ is a positive integer, by comonotone additivity we get
\[
G(n\cdot\x,\mu_r)=G(\overbrace{\x,\ldots,\x}^{\textnormal{n times}},\mu_r)=n\cdot G(\x,\mu_r), 
\qquad \textnormal{for every }\x\in\mathcal I^n.
\]
If $a=n/m\in\raz^+$ is a positive razional number, with $n,m\in\nat$ we get 
\[
n\cdot G(\x,\mu_r)= G(n\cdot\x,\mu_r)=G(\frac{nm}{m}\cdot\x,\mu_r)=m\cdot G(\frac{n}{m}\cdot\x,\mu_r),
\qquad \textnormal{for every }\x\in\mathcal I^n.
\]
Finally, for $a\in \rea^+\setminus\raz^+$ it is sufficient to consider two sequences of razional numbers convergent to 
$a$,  $\{a_i^-\}$ and $\{a_i^+\}$ such that 
$a_1^-<a_2^-\ldots<a<\ldots a_2^+<a_1^+$ and using monotonicity of $G$ we get that 
$G(a\cdot \x,\mu_r)=a\cdot G(\x,\mu_r)$ for every $\x\in\mathcal I^n.$\\
Regarding idempotency, if $a\in\rea^+$ we get $G(a\cdot \textbf{1}_{(N,N)})=a\cdot G(\textbf{1}_{(N,N)})=a$. 
By comonotone additivity 
$0=G(0\cdot\textbf{1}_{(N,N)})$ 
and 
$0=G((a-a)\textbf{1}_{(N,N)})= G(a\cdot\textbf{1}_{(N,N)})+  G(-a\cdot\textbf{1}_{(N,N)})=a+ G(-a\cdot\textbf{1}_{(N,N)})$, 
thus 
$G(-a\cdot\textbf{1}_{(N,N)})=-a$.\\
The hypotheses of theorem ensure that the 
\begin{equation}
\mu_r(A,B)=G\left(\textbf{1}_{(A,B)}\right)\ \forall(A,B) \in \mathcal Q  
\label{eq:intervalcap}
\end{equation}
defines an interval capacity.
Indeed: 
$\mu_r(N,N)=G(\textbf{1}_{(N,N)})=1$; 
$\mu_r(\emptyset,\emptyset)=G(\textbf{1}_{(\emptyset,\emptyset)})=0$, since by comonotone additivity 
$G(\textbf{1}_{(\emptyset,\emptyset)})=G(\textbf{1}_{(\emptyset,\emptyset)}+\textbf{1}_{(\emptyset,\emptyset)})=
G(\textbf{1}_{(\emptyset,\emptyset)}) +G(\textbf{1}_{(\emptyset,\emptyset)})$ and thus $G(\textbf{1}_{(\emptyset,\emptyset)})=0$; 
for all $(A,B), (C,D) \in \mathcal Q$ such that $(A, B) \subseteq (C,D)$, $\mu_r(A,B) \leq \mu_r(C,D)$ follows by monotonicity of $G$.
Let  
$\x=\left(\left[\lf _1,\rg _1\right],\ldots,\left[\lf _n,\rg _n\right]\right)$ be a vector and  
$(\cdot):\left\{1,\ldots,2n\right\}\rightarrow \left\{1,\ldots,2n\right\}$ be a permutation 
such that 
$x_{(1)}\leq x_{(2)}\leq\ldots\leq x_{(2n)}$.
For all $i=1,\ldots,2n$ define 
$\underline{A}_{(i)}=\left\{i\in N\ |\ \lf_i\geq x_{(i)}\right\}$ and  
$\overline{A}_{(i)}=\left\{i\in N\ |\ \rg_i\geq x_{(i)}\right\}$.  
Clearly $\left(\underline{A}_{(i)},\overline{A}_{(i)}\right)\in\mathcal Q$ and, since 
$\underline{A}_{(i+1)}\subseteq \underline{A}_{(i)}$ and 
$\overline{A}_{(i+1)}\subseteq \overline{A}_{(i)}$, 
then the the vectors  
$\textbf{1}_{\left(\underline{A}_{(i)},\overline{A}_{(i)}\right)}$ 
are comonotone for all $i=1,\ldots,2n$. 
The vector $\x$ can be rewritten as sum of comonotone vectors (take $x_{(0)}=0$):
\begin{equation}
\x=\sum_{i=1}^{2n}{\left[x_{(i)}-x_{(i-1)}\right]\cdot \textbf{1}_{\left(\underline{A}_{(i)},\overline{A}_{(i)}\right)}}.
\label{eq:fsum}
\end{equation}
Finally, the proof follows from \eqref{eq:fsum} by using, respectively, comonotone additivity, homogeneity of $G$ and definition of the interval capacity $\mu_r$ according to \eqref{eq:intervalcap}:
\[
G(\x)=
G\left(
\sum_{i=1}^{2n}{\left[x_{(i)}-x_{(i-1)}\right]\cdot \textbf{1}_{\left(\underline{A}_{(i)},\overline{A}_{(i)}\right)}}
\right)
=
\sum_{i=1}^{2n}
G\left(
\left[x_{(i)}-x_{(i-1)}\right]\cdot 
\textbf{1}_{\left(\underline{A}_{(i)},\overline{A}_{(i)}\right)}\right)=
\]
\[
=
\sum_{i=1}^{2n}{\left[x_{(i)}-x_{(i-1)}\right]\cdot 
G\left(\textbf{1}_{\left(\underline{A}_{(i)},\overline{A}_{(i)}\right)}\right)}
=\sum_{i=1}^{2n}{\left[x_{(i)}-x_{(i-1)}\right]\cdot 
\mu_r\left(\underline{A}_{(i)},\overline{A}_{(i)}\right)}=
Ch_r(\x,\mu_r).
\]
\end{proof}

\section{The RCI and M{\"o}bius inverse}
The following proposition gives the closed formula of the M{\"o}bius inverse \cite{rota1964foundations} of a function on $\mathcal Q$.
\begin{pro}
Suppose $f,g:\mathcal Q\rightarrow\rea$ are two real valued functions on $\mathcal Q$. Then 
\begin{equation}
f(A,B)=\sum_{\substack{(C,D)\in\mathcal Q \\(C,D)\subseteq(A,B)}
}
{\hspace{-5mm}g(C,D)}\qquad\textnormal{for all}\quad (A,B)\in\mathcal Q
\label{eq:mobius1}
\end{equation}
if and only if 
\begin{equation}
g(A,B)=
\sum_{\emptyset\subseteq X\subseteq A}
{
\left(-1\right)^{\abs{X}}
\hspace{-10mm}\sum_{\substack{(C,D)\in\mathcal Q \\(C,D)\subseteq (A\setminus X,B\setminus X)}}{\hspace{-8mm}\left(-1\right)^{\abs{B\setminus A}-\abs{D\setminus C}}f(C,D)}
} 
\qquad\textnormal{for all}\quad (A,B)\in\mathcal Q.
\label{eq:mobius2}
\end{equation}
\label{pro:mobius}
\end{pro}
\begin{proof}
See Appendix
\end{proof}
\begin{rem}
By setting for all $X\subseteq A\subseteq N$ and for all $(A,B)\in\mathcal Q$ 
\[
g^*(A\setminus X,B\setminus X)=\sum_{(C,D)\subseteq (A\setminus X,B\setminus X)}{\left(-1\right)^{\abs{B\setminus A}-\abs{D\setminus C}}f(C,D)},
\]
thus equation \eqref{eq:mobius2} can be rewritten as 
\[
g(A,B)=\sum_{\emptyset\subseteq X\subseteq A}{\left(-1\right)^{\abs{X}}g^*(A\setminus X,B\setminus X)}.
\]
\end{rem}
\begin{rem}
Let us apply proposition \ref{pro:mobius} to $\mathcal Q_0=\{(\emptyset, B))\ |\ B\subseteq N\}\subseteq \mathcal Q$ which we identify with $2^N$. 
Thus we obtain the well known result, applied to functions $f,g:2^N\rightarrow \rea$,  
\begin{equation}
f(B)=\sum_{D\subseteq B}g(D) \qquad\textnormal{for all}\quad B\in 2^N
\label{eq:mobus10}
\end{equation}
if and only if 
\begin{equation}
g(B)=
\sum_{D\subseteq B}\left(-1\right)^{|B\setminus D|}f(D) \qquad\textnormal{for all}\quad B\in 2^N.
\label{eq:mobius20}
\end{equation}
\end{rem}
\noindent The first of the two following propositions characterizes an interval capacity by means of its M{\"o}bius inverse. 
The second one allows the RCI with respect to an interval capacity to be rewritten using the M{\"o}bius inverse of such an interval capacity.

\begin{pro}
$\mu_r:\mathcal Q\rightarrow \rea$ is an interval capacity if and only if 
its M{\"o}bius inverse $\mu_r:\mathcal Q\rightarrow \rea$ satisfies:
\begin{enumerate}
	\item $m\left(\emptyset,\emptyset\right)=0$;
	\item $\sum_{(A,B)\in\mathcal Q}{m(A,B)}=1$;
	\item $\sum_{\left\{a\right\}\subseteq C\subseteq A}{\sum_{C\subseteq D\subseteq B}}{m(C,D)}\geq 0$,  
	  $\forall a\in A\subseteq B \in 2^N$;
	\item $\sum_{\left\{b\right\}\subseteq D\subseteq B}{\sum_{C\subseteq A\cap D}}{m(C,D)}\geq 0$,  
	  $\forall b\in B\supseteq A \in 2^N$.
\end{enumerate}
\label{pro:mobius1}
\end{pro}
\begin{proof}
See Appendix
\end{proof}

\begin{pro}
Let $\mu_r:\mathcal Q \rightarrow [0,1]$ be an interval capacity and let $m:\mathcal Q \rightarrow [0,1]$ 
be its M{\"o}bius inverse, then for all $\x\in\mathcal I^n$ 
\begin{equation}
Ch_r(\x,\mu_r)=\sum_{(A,B)\in\mathcal Q}
m(A,B)\bigwedge\left\{\underset{i\in A}{\bigwedge}\lf_i, \underset{i\in B}{\bigwedge}\rg_i\right\}.
\label{eq:mobius3}
\end{equation}
\label{pro:mobius2}
\end{pro}
\begin{proof}
For all $\x\in\mathcal I^n$,
\begin{eqnarray}
Ch_r(\x,\mu_r)& = &\sum_{i=1}^{2n}{x_{(i)}\left[\mu_r\left(A_{(i)},B_{(i)}\right)-\mu_r\left(A_{(i+1)},B_{(i+1)}\right)\right]}\nonumber=  
\\&=& 
\sum_{i=1}^{2n}x_{(i)}\sum_{(A,B)\subseteq \left(A_{(i)},B_{(i)}\right)\setminus \left(A_{(i+1)},B_{(i+1)}\right)}m(A,B)
=\sum_{(A,B)\in\mathcal Q}
m(A,B)\bigwedge\left\{\underset{i\in A}{\bigwedge}\lf_i, \underset{i\in B}{\bigwedge}\rg_i\right\}.
\end{eqnarray}
\end{proof}
\begin{rem}
Note that the term $\underset{i\in B}{\bigwedge}\rg_i$ can also be written $\underset{i\in B\setminus A}{\bigwedge}\rg_i$ and can have an influence. 
See, e.g., the following example: $N=\{1,2\}$, $(A,B)=(1,12)$, $x=([3,4],[1,2])$.
In this case, by applying the \eqref{eq:mobius3} the term $m\left(\{1\},\{1,2\}\right)$ must be multiplied by $2=\min\{3,4,2\}=\min\{3,2\}$.
\end{rem}
\noindent Using previous proposition the RCI assumes a linear expression with respect to 
the interval-measure.

\begin{cor}
There exist functions $f_{(A,B)}:\rea^n\rightarrow\rea$, $(A,B)\in\mathcal Q$ such that
\begin{equation}
Ch_r(\x,\mu_r)=\sum_{(A,B)\in\mathcal Q}
\mu_r(A,B)f_{(A,B)}(\x).
\label{eq:mobius4}
\end{equation}
\label{cor1}
\end{cor}
\begin{proof}
Indeed, using the \eqref{eq:mobius2}  
\begin{equation}
m(A,B)=
\sum_{\emptyset\subseteq X\subseteq A}
{
\left(-1\right)^{\abs{X}}
\hspace{-10mm}\sum_{\substack{(C,D)\in\mathcal Q \\(C,D)\subseteq (A\setminus X,B\setminus X)}}{\hspace{-8mm}\left(-1\right)^{\abs{B\setminus A}-\abs{D\setminus C}}\mu_r(C,D)}
} 
\qquad\textnormal{for all}\quad (A,B)\in\mathcal Q.
\label{}
\end{equation}
in the \eqref{eq:mobius3}, the \eqref{eq:mobius4} is verified with 
\begin{equation}
f_{(A,B)}(\x)=\sum_{\emptyset\subseteq X\subseteq N\setminus A}\left(-1\right)^{|X|}
\sum_{(A\cup X,B\cup X)  }\left(-1\right)^{|B\setminus A|}\bigwedge\left\{\underset{i\in A\cup X}{\bigwedge}\lf_i, \underset{i\in B\cup X}{\bigwedge}\rg_i\right\}
\label{eq:}
\end{equation}
\end{proof}

\section{The robust Sugeno and Shilkret integrals}
Let us consider a set of criteria $N=\left\{1,2,\ldots,n\right\}$ and a set of alternatives
$A=\left\{\x,\y,\z,\ldots\right\}$ to be evaluated, on each criterion, on the scale $[0,1]$. 
Thus each $\x\in A$ can be identified with a score vector $\x=\left(x_1,\ldots,x_n\right)\in\left[0,1\right]^n$,  
whose $i^{th}$ component, $x_i$, represents the evaluation of $\x$ with respect to the $i^{th}$ criterion. 

\begin{Def}
The Sugeno Integral \cite{sugeno1974theory} of 
$\x=\left(x_1,\ldots,x_n\right)\in\left[0,1\right]^n$ 
with respect to the capacity $\nu:2^{[0,1]}\rightarrow [0,1]$ is 
\begin{equation}
S\left(\x,\nu\right)=\bigvee_{i\in N}\bigwedge \left\{x_{(i)}, \nu\left(A_{(i)}\right)\right\},
\label{eq:S1}
\end{equation}
being $(\cdot):N\rightarrow N$ an indexes permutation such that 
$x_{(1)}\leq\ldots\leq x_{(n)}$ and  
$A_{(i)}=\left\{(i),\ldots,(n)\right\}$, $i=1,\ldots,n$.
\end{Def}
\noindent It follows from the definition that 
$S\left(\x,\nu\right)\in\left\{x_1,\ldots,x_n\right\}\bigcup\left\{\nu(A)\ |\ A\subseteq N\right\}$.
Moreover the Sugeno integral can also be computed if the elements of the set 
$S\left(\x,\nu\right)\in\left\{x_1,\ldots,x_n\right\}\bigcup\left\{\nu(A)\ |\ A\subseteq N\right\}$
are just ranked on an ordinal scale. \\
The \eqref{eq:S1} involves $n$ terms but requests a permutation.
An equivalent formulation (see \cite{marichal2001axiomatic}) involves $2^n$ terms but does not request a permutation.
\begin{equation}
S\left(\x,\nu\right)=\bigvee_{A\subseteq N} \bigwedge
\left\{
\nu\left(A\right),\bigwedge_{i\in A}x_i
\right\}.
\label{eq:S2}
\end{equation}
Now, suppose that for every $\x\in A$, we have, on each criterion, a numerical imprecise evaluation on the scale $[0,1]$. 
Specifically, suppose that for each $i\in N$ we know a range $\left[\lf _i,\rg _i\right]\subseteq [0,1]$ containing the exact evaluation of $\x$ with respect to $i$. 
Thus, being 
$\mathcal I_{[0,1]}=\left\{[a,b]\ |\  a,b\in [0,1],\  a\le b\right\}$ the set of bounded and closed subintervals of $[0,1]$,  
any alternative $\x$ can be identified with a score vector  
\begin{equation}
\x=\left(\left[\lf _1,\rg _1\right],\ldots,\ \left[\lf _i,\rg _i\right],\ldots,\left[\lf _n,\rg _n\right]\right)\in\mathcal I_{[0,1]}^n,
\label{eq:vectorr}
\end{equation}
whose $i^{th}$ component, 
$x_i=[\lf_i,\rg_i]$, 
is the interval containing the evaluation of $\x$ with respect to the $i^{th}$ criterion. 
Vectors of $[0,1]^n$ are considered elements of $\mathcal I_{[0,1]}^n$ by identifying each $x\in [0,1]$ with the degenerate interval $[x,x]=\{x\}$.
We associate to every 
$\x=\left(\left[\lf _1,\rg _1\right],\ldots,\left[\lf _n,\rg _n\right]\right)\in \mathcal I^n$ the vector 
$\underline{\x}=\left(\lf_1,\ldots,\lf_n\right)$ 
of all the worst (or pessimistic) evaluations and 
the vector 
$\overline{\x}=\left(\rg_1,\ldots,\rg_n\right)$ 
of all the best (or optimistic) evaluations on each criterion.

\begin{Def}
The Robust Sugeno Integral (RSI) of $\x$ with respect to the interval capacity $\mu_r$ is 

\begin{equation}
S_r\left(\x,\mu_r\right)=\bigvee_{(A,B)\in \mathcal Q} \bigwedge\left\{\underset{i\in A}{\bigwedge}\ \underline{x}_i,\  
\underset{i\in B\minus A}{\bigwedge}\ \overline{x}_i\ ,\ 
\mu_r\left(A,B\right)\right\}.
\label{eq:RSI2}
\end{equation}
\end{Def}
\noindent It follows from the definition that 
$S_r\left(\x,\mu_r\right)\in\left\{\lf_1,\ldots,\lf_n\right\}\bigcup\left\{\rg_1,\ldots,\rg_n\right\}\bigcup\left\{\mu_r(A,B)\ |\ (A,B)\in \mathcal Q\right\}$.
Moreover the RSI can also be computed if the elements of this set 
are just ranked on an ordinal scale. \\
\noindent The \eqref{eq:RSI2} involves $|\mathcal Q|=3^n$ terms. 
An alternative formulation of the RSI implies some additional notations. 
We identify every vector  
$\x=\left(\left[\lf _1,\rg _1\right],\ldots,\left[\lf _n,\rg _n\right]\right)\in \mathcal I_{[0,1]}^n$ 
with the vector 
$\x^*=(x_1,\ldots,x_{2n})\in [0,1]^{2n}$ 
defined according to \eqref{eq:ident}.  
Let 
$(\cdot):\left\{1,\ldots,2n\right\}\rightarrow \left\{1,\ldots,2n\right\}$
be a permutation of indices such that  
$x_{(1)}\leq x_{(2)}\leq\ldots\leq x_{(2n)}$
and for all $i=1,\ldots,2n$ let us define 
$A_{(i)}=\left\{j\in N\ |\ \lf_j\geq x_{(i)}\right\}$ and 
$B_{(i)}=\left\{j\in N\ |\ \rg_j\geq x_{(i)}\right\}$.
Thus, the RSI of $\x$ with respect to the interval capacity $\mu_r$ is:
\begin{equation}
S_r\left(\x,\mu_r\right)=\bigvee_{i\in \left\{1,\ldots , 2n\right\}} \bigwedge\left\{x_{(i)}, \mu_r\left(A_{(i)},B_{(i)}\right)\right\}.
\label{eq:RSI1}
\end{equation}
This requests  $2n$ terms and an indices permutation.\\
We close this subsection with two illustrative examples. 
The first just shows the equivalence of formulation \eqref{eq:RSI1} and \eqref{eq:RSI2}, the scale $[0,1]$ is substituted with the scale $[0,10]$.
The second is applied to a problem of students evaluations and the scale $[0,1]$ is substituted with the scale $[0,30]$.\\ 
\textbf{Example 1}\\
\noindent Let us suppose that $N=\left\{1,2\right\}$ and consider $\x=\left(\left[5,9\right],\left[2, 4\right]\right)$. 
Let be given the following interval capacity on $\mathcal Q$:
\[
\mu_r\left(\emptyset,\emptyset\right)=0, \ 
\mu_r\left(\emptyset,\left\{1\right\}\right)=3, \
\mu_r\left(\emptyset,\left\{2\right\}\right)=2, \
\mu_r\left(\emptyset,N\right)=5, \
\mu_r\left(\left\{1\right\},\left\{1\right\}\right)=4, \]
\[
\mu_r\left(\left\{1\right\},N\right)=6, \
\mu_r\left(\left\{2\right\},\left\{2\right\}\right)=4, \
\mu_r\left(\left\{2\right\},N\right)=7, \
\mu_r\left(N,N\right)=10.
\]
It follows that 
\[
\mu_r\left(A_2,B_2\right)=\mu_r\left(N,N\right)=1,\ 
\mu_r\left(A_4,B_4\right)=\mu_r\left(\left\{1\right\},N\right)=6,\]
\[
\mu_r\left(A_5,B_5\right)=\mu_r\left(\left\{1\right\},\left\{1\right\}\right)=5,\ 
\mu_r\left(A_9,B_9\right)=\mu_r\left(\emptyset,\left\{1\right\}\right)=3.
\]
By using the \eqref{eq:RSI1} we get
\[
S_r\left(\x,\mu_r\right)=\max\left\{\min\left\{2,10\right\}, \min\left\{4,6\right\},\min\left\{5,4\right\}, \min\left\{9,3\right\}\right\}=\max\left\{2,4,4,3\right\}=4.
\]
Alternatively, we can use the \eqref{eq:RSI1}
\[
S_r\left(\x,\mu_r\right)=\max\left\{0,\min\left\{3,9\right\}, \min\left\{2,4\right\}, 
\min\left\{5,4\right\},
\min\left\{4,5,9\right\}, 
\right.
\]
\[\left., \min\left\{6,5,4\right\}, 
\min\left\{4,2,4\right\},
\min\left\{7,2,4\right\},
\min\left\{10,2,4\right\}
\right\}=4.
\]
\textbf{Example 2}\\
Suppose we need to evaluate a university student in four economic subjects, $N=\left\{m_1,m_2,m_3,m_4\right\}$ of which  $\left\{m_1,m_2\right\}$ belong to the subcategory of microeconomic.
We suppose that the student is evaluated on each subject by a 30 point scale, allowing interval (imprecise) evaluations. 
Let us consider the vector $E(Student)=E(S)$ containing the single evaluation in each subject $E(m_i)$:
\begin{equation}
E(S)=\left(E(m_1),E(m_2)\, E(m_3), E(m_4)\right)=
\left(\left[26,30\right],\left[28,30\right],\left[24,27\right],\left[23,27\right]\right)
\label{eq:exsug1}
\end{equation}
In order to compute the RSI of $E(S)$ we have to specify some values of an interval capacity defined on $\mathcal Q$. 
For example the following:
\begin{eqnarray}
\mu_r\left(N,N\right)=30,\ 
\mu_r\left(\left\{m_1,m_2,m_3\right\},N\right)=29,\ 
\mu_r\left(\left\{m_1,m_2\right\},N\right)=28,\ \nonumber\\
\mu_r\left(\left\{m_2\right\},N\right)=24,\ 
\mu_r\left(\left\{m_2\right\},\left\{m_1,m_2\right\}\right)=23
\mu_r\left(\emptyset,\left\{m_1,m_2\right\}\right)=20.
\label{eq:exsug2}
\end{eqnarray}
These weights reflect the fact that we retain the microeconomic subcategory $\left\{m_1,m_2\right\}$ particularly important. 
Indeed when $\left\{m_1,m_2\right\}$ is not included on $A$ the weight assigned to $(A,B)$ is small.
The question is: how much should be globally evaluated the student in accordance with the partials evaluations \eqref{eq:exsug1} and \eqref{eq:exsug2}?
Using the RSI, equation \eqref{eq:RSI2}, such a student should be evaluated 
$$S_r\left(S,\mu_r\right)= \bigvee\left\{\bigwedge\left\{23,30\right\},\ \bigwedge\left\{24,29\right\},
\ \bigwedge\left\{26,28\right\}\ \bigwedge\left\{27,24\right\}\ \bigwedge\left\{28,23\right\}\ \bigwedge\left\{30,20\right\}\right\}=26.$$
In this case we cannot assign a greater evaluation, due to the pessimistic evaluation of the student in the relevant subject $m_1$.\\
For nonnegative valued alternative, another famous integral useful to aggregate criteria evaluations is the Shilkret integral \cite{shilkret1971maxitive}.  
\begin{Def}
The  Shilkret integral \cite{shilkret1971maxitive} of a vector $\textbf{\textit{x}}=\left(x_1, \ldots, x_n\right)\in \left[0,1\right]^n$ with respect to the capacity $\nu$ is given by
\begin{equation}
Sh(\textbf{\textit{x}},\nu)=\bigvee_{i \in N}\left\{x_i\cdot\nu(\{j \in N: x_j \ge x_i\}\right\}.
\label{shilkret}
\end{equation}
\end{Def}
\noindent For interval evaluations on the criteria, the Shilkret integral can be computed with respect to an interval capacity. 
Let us define 
$\mathcal I{[0,1]}^n=\left\{[a,b]\ |\  a,b\in\rea,\  0\le a\le b\le 1\right\}$ thus we have the following

\begin{Def}
The robust Shilkret integral of $\x\in\mathcal I_{[0,1]}^n$ with respect to the interval capacity $\mu_r$ is  
\[
Sh_r(x,\mu_r)=\underset{(A,B)\in \mathcal Q}{\bigvee}\left\{\bigwedge\left\{\underset{i\in A}{\bigwedge}\lf_i, \underset{i\in B}{\bigwedge}\rg_i\right\}\cdot\mu_r(A,B)\right\}.
\]
\end{Def}

\section{Other robust integrals}
What we have done regarding the Choquet, Shilkret and the Sugeno integrals can be extended to other integrals. 
Recently, in the context of multiple criteria decision analysis, the literature on fuzzy integrals has increased very fast. 
An interesting line of research is that of bipolar fuzzy integrals: the bipolar Choquet integral has been proposed in \cite{grabisch2005biI,grabisch2005bi,greco2002bipolar} and the bipolar Shilkret and 
Sugeno integrals have been proposed in \cite{greco2012FSS}. 
Here we propose the generalization of the bipolar Choquet integral to the case of interval evaluations.   
Let us consider the set  
$$
\mathcal Q_b=
\left\{\left(A^+,B^+,A^-,B^-\right)\ |\ A^+\subseteq B^+\subseteq N,\  N\supseteq A^-\supseteq B^-\ \textnormal{and}\ B^+\cap A^-=\emptyset\right\}.
$$ 

\begin{Def} 
A function $\mu^b_r: \mathcal Q_b \rightarrow [-1,1]$ is a \textit{bipolar interval-capacity} on $\mathcal Q_b$  if
\begin{itemize}
	\item $\mu_r^b(\emptyset,\emptyset,\emptyset,  \emptyset)=0$, $\mu_r^b(N,N,\emptyset,\emptyset)=1$ and $\mu_r^b(\emptyset,\emptyset,N,N)=-1$;
	\item $\mu_r^b\left(A_1^+,B_1^+,A_1^-,B_1^-\right) \leq \mu_r^b\left(A_2^+,B:_2^+,A_2^-,B_2^-\right)$ for all 
		$\left(A_1^+,B_1^+,A_1^-,B_1^-\right), \left(A_2^+,B:_2^+,A_2^-,B_2^-\right) \in \mathcal Q_b$  such that 
		$A_1^+\subseteq A_2^+$, $B_1^+\subseteq B_2^+$, $A_1^-\supseteq A_2^-$ and $B_1^-\supseteq B_2^-$.
\end{itemize}
\label{def:bipolar interval capacity}
\end{Def}

\begin{Def}
The bipolar Robust Choquet Integral (bRCI) of 
$\x=\left(\left[\lf _1,\rg _1\right],\ldots,\left[\lf _n,\rg _n\right]\right)\in \mathcal I^n$ with respect to a bipolar interval capacity $\mu_r^b:2^N\rightarrow [0,1]$ is given by:
\begin{equation}
Ch_r^b\left(\x,\mu_rb\right)=: 
\int_{-\infty }^\infty \mu_r^G(\{i\ |\  \lf_i > t\} ,\{i \  |\ \rg_i> t \}, \{i\ |\  \lf_i <-t \} ,\{i \  |\ \rg_i<-t  \})dt.
\label{eq:bRCI}
\end{equation}
\end{Def}
A further generalization in the field of fuzzy integrals is that of level dependent integrals.
This line of research has lead to the definition of the level dependent Choquet integral and the bipolar level dependent Choquet integral \cite{greco2011choquet}, the level dependent Shilkret integral \cite{bodjanova2009sugeno}, the level dependent Sugeno integral \cite{mesiar2009level}. 
In \cite{greco2011choquet} the \textit{generalized Choquet integral} is defined with respect to a level dependent capacity. 
Also the RCI can be generalized in this sense. 
\begin{Def}
Let $(\alpha,\beta)\subseteq\rea$ be any possible interval of the real line. 
A generalized interval capacity is a function $\mu_r^G:\mathcal Q\times(\alpha,\beta)\rightarrow[0,1]$ such that
\begin{enumerate}
	\item for all $t\in(\alpha,\beta)$ and $(A,B)\subseteq(C,D)\in\mathcal Q$, 
	$\mu_r^G\left(\left(A,B\right),t\right)\leq\mu_r^G\left(\left(C,D\right),t\right)$
	\item for all $t\in(\alpha,\beta),\ \mu_r^G\left(\left(\emptyset,\emptyset\right),t\right)=0$ and 
	$\mu_r^G\left(\left(N,N\right),t\right)=1$
	\item  for all $(A,B)\in\mathcal Q,\ \mu_r^G\left(\left(A,B\right),t\right)$ considered as a function with respect to $t$ is Lebesgue measurable.
\end{enumerate}
\end{Def}
\begin{Def}
The generalized Robust Choquet Integral (RCIg) of $\x\in \left(\mathcal I_{(\alpha,\beta)}\right)^n$ with respect to a generalized interval capacity $\mu_r^G:\mathcal Q\times(\alpha,\beta)\rightarrow[0,1]$ is given by:
\begin{equation}
Ch_r^G\left(\x,\mu_r\right)=: 
\int_{\min\left\{\lf_1,\ldots,\lf_n\right\}}^\infty \left(\mu_r^G(\{i \in N\  |\  \lf_i \geq t \},\{i \in N\ |\ \rg_i\geq t \}),t \right)dt\ +\min\left\{\lf_1,\ldots,\lf_n\right\}.
\label{eq:RCIG}
\end{equation}
\[\]
\end{Def}
\noindent The RCIg can be characterized by the following three properties: Idempotency, Monotonicity and Tail Independence (see \cite{green1988ordinal}).
The following example illustrates the Tail independence in the framework of imprecise evaluations.
Let be: 
$\x=\left([1,3],[0,6],[2,3],[4,5]\right)$, 
$\textbf{y}=\left([1,3],[0,4],[2,3],[3,7]\right)$, 
$\textbf{w}=\left([0,2],[1,6],[0,2],[4,5]\right)$, 
$\textbf{z}=\left([0,2],[1,4],[0,2],[3,7]\right)$. 
Then, given the aggregation function $G$, Tail Independence means 
$$G(x)-G(y)=G(w)-G(z).$$
That is the classical tail independence, applied on the interval extremes $\lf_i$ and $\rg_i$ permuted.\\
The last example we wish to provide is the generalization of the Concave Integral, proposed in \cite{lehrer2008concave}. 
Consider the set $\mathcal I_+=\left\{[a,b]\ |\  a,b\in\rea \  0\le a\le b\right\}$.  

\begin{Def}
The Robust Concave Integral of a nonnegative interval valued alternative $\x\in\mathcal I_+^n$ 
with respect to the interval capacity $\mu_r$ is 
\begin{equation}
\int^{cav}{\x d\mu_r}=\bigvee\left\{\sum_{\left(A,B\right)\in\mathcal Q}{\alpha_{(A,B)}\mu_r(A,B)}
\ ;\ 
\sum_{\left(A,B\right)\in\mathcal Q}{\alpha_{(A,B)}\textbf{1}_{(A,B)}}=\x,\ \alpha_{(A,B)}\geq 0
\right\}.
\label{eq:RcavI}
\end{equation}
\end{Def}

\noindent Obviously, if on every criterion $\x$ receives an exact evaluation, thus the \eqref{eq:RcavI} reduces to the Concave Integral of $\x\in\rea_+^n$ with respect to the capacity $\nu(A)=\mu_r(A,A)$.

\section{Generalizing the concept of interval to m-points interval}

In \cite{ozturk2011representing} the concept of interval has been generalized (allowing the presence of more than two points). \\
We can image that on every of the $n$ criteria an alternative $\x$ is evaluated m times, so that this alternative can be identified with a vector of score vectors $\x=(x_1,\ldots,x_n)$ being for all $i=1,\ldots,n$
\[
x_i=
\left(
f_1(x_i),\ldots f_m(x_i)
\right)
\qquad\textnormal{with}\qquad
\ 
f_j(x_i)\leq f_{j+1}(x_i)
\ 
\textnormal{for all}
\ 
j=1,\ldots,m-1.
\]
\noindent For example, the case m=3 corresponds to have on each criterion a pessimistic, a realistic and an optimistic evaluation.\\
The idea to extend the RCI to the case of m-interval based evaluation is simple.
Let us define
\[
\mathcal Q_m=\left\{(A_1,\ldots,A_m)\ |\ A_1\subseteq A_2\ldots\subseteq A_m\subseteq N\right\}.
\]
\begin{Def}
An m-interval capacity is a function $\mu_m:\mathcal Q_m\rightarrow [0,1]$ such that
\begin{itemize}
	\item $\mu_m(\emptyset,\ldots,\emptyset)=0$,
	\item $\mu_m(N,\ldots,N)=1$,
	\item $\mu_m(A_1,\ldots,A_m)\leq\mu_m(B_1,\ldots,B_m)$, whenever $A_i\subseteq B_i\subseteq N,\ \forall i=1,\ldots,m$.
\end{itemize}
\end{Def}

\begin{Def}
The Robust Choquet Integral of $\x$ (m-points interval valued) w.r.t. the m-interval capacity $\mu_m$ is 
\small{\begin{equation}
Ch_r(\x,\mu_m)=\int_{\min_if_1(x_i)}^{\max_if_m(x_i)}
\mu_m
\left(
\{
j\in N\ |\ f_1(x_j)\geq t
\},
\ldots,\{j\in N\ |\ f_m(x_j)\geq t\}dt
\right)
+\min_if_1(x_i).
\label{eq:opop}
\end{equation}
}
\end{Def}

\section{Conclusions}
In this paper we have faced the question regarding the aggregation of interval evaluations of an alternative on various criteria into a single overall evaluation. 
To this scope we have introduced the concept of interval capacity which allows for a quite natural generalizations of the classical Choquet Shilkret and Sugeno integrals to the case of interval evaluations. 
We called these generalizations robust integrals. 
Our analysis shows that, when the interval evaluations collapse into exact evaluations, our definitions of robust integrals collapse into the original definitions.
Situations where we meet imprecise evaluations are very common in the real life (we have provided realistic examples), so the aim of this paper is to cover the existing gap in the literature for the aggregations of such data.

\newpage

\section{Appendix}\label{appendix}
In order to prove proposition \ref{pro:mobius}, we need some preliminary lemmas. 
\subsection{Preliminary lemmas}
The following two lemmas have been proved in \cite{shafer1976mathematical} (see also \cite{chateauneuf1989some})
\begin{lem}
If $A$ is a finite set then
\begin{equation}
\sum_{B\subseteq A}{\left(-1\right)^{|B|}}=\left\{\begin{array}{ccl}
1& & \textit{if  } A=\emptyset\\
0 && \textit{otherwise.}
\end{array}\right.
\label{}
\end{equation}
\label{lemma21}
\end{lem}

\begin{lem}
If $A$ is a finite set and $B\subseteq A$ then
\begin{equation}
\sum_{B\subseteq C\subseteq A}{\left(-1\right)^{|C|}}=\left\{\begin{array}{ccl}
\left(-1\right)^{|A|}& & \textit{if  } A=B\\
0 && \textit{otherwise.}
\end{array}\right.
\label{}
\end{equation}
\label{lemma22}
\end{lem}

\noindent With these results we are able to prove the following additional lemmas

\begin{lem}
For all $(A,B)\in\mathcal Q$ 
\begin{equation}
\sum_{
\substack{
(C,D)\in\mathcal Q \\ 
(C,D)\subseteq (A,B)
}
}{\left(-1\right)^{|D|}}=\left\{\begin{array}{ccl}
\left(-1\right)^{|B|}& & \textit{if  } A=B\\
0 && \textit{otherwise.}
\end{array}\right.
\label{eq:ui}
\end{equation}
\label{lemma21*}
\end{lem}

\begin{proof}
\[
\sum_{
\substack{
(C,D)\in\mathcal Q \\ 
(C,D)\subseteq (A,B)
}
}{\left(-1\right)^{|D|}}=\sum_{X\subseteq A}\sum_{Y\subseteq B\setminus X}\left(-1\right)^{|Y|}=\textnormal{lemma \ref{lemma21}}=
\left(-1\right)^{|A|}\sum_{Y\subseteq B\setminus X}\left(-1\right)^{|Y|}
+
\sum_{X\subseteq A}\sum_{Y\subset B\setminus X}\left(-1\right)^{|Y|}=
\]
\[
=
\left(-1\right)^{|A|}\sum_{Y\subseteq B\setminus X}\left(-1\right)^{|Y|}
=\textnormal{lemma \ref{lemma21}}=
\left\{\begin{array}{ccl}
\left(-1\right)^{|B|}& & \textit{if  } A=B\\
0 && \textit{otherwise.}
\end{array}\right.
\]
\end{proof}

\begin{rem}
If $A=\emptyset$ then lemma \ref{lemma21*} coincides with lemma \ref{lemma21}.
\end{rem}

\begin{cor}
For all $(A,B)\in\mathcal Q$ 
\begin{equation}
\sum_{
\substack{
(C,D)\in\mathcal Q \\ 
(C,D)\subseteq (A,B)
}
}{\left(-1\right)^{|B\setminus D|}}=\left\{\begin{array}{ccl}
1& & \textit{if  } A=B\\
0 && \textit{otherwise.}
\end{array}\right.
\label{eq:cor21*}
\end{equation}
\label{cor21*}
\end{cor}

\begin{proof}
For all $(A,B)\in\mathcal Q$ 
\[
\sum_{
\substack{
(C,D)\in\mathcal Q \\ 
(C,D)\subseteq (A,B)
}
}{\left(-1\right)^{|B\setminus D|}}=
\left(-1\right)^{|B|}\sum_{
\substack{
(C,D)\in\mathcal Q \\ 
(C,D)\subseteq (A,B)
}
}{\left(-1\right)^{|D|}}=
\textnormal{lemma \ref{lemma21*}}=
\left\{\begin{array}{ccl}
1& & \textit{if  } A=B\\
0 && \textit{otherwise.}
\end{array}\right.
\]
\end{proof}

\begin{lem}
Suppose that $(C,D), (A,B)\in\mathcal Q$ with $(C,D)\subseteq(A,B)$, then
\small{\begin{equation}
\sum_{\substack{(X,Y)\in\mathcal Q \\ (C,D)\subseteq (X,Y)\subseteq (A,B)}}
\left(-1\right)^{|Y|}=
\left\{
\begin{array}{cl}
\left(-1\right)^{|B|}|2^{\left(A\cap D\right)\setminus C}|
& \textnormal{if}\quad A\cup D=B\quad\textnormal{i.e}\quad B\setminus \left(A\cup D\right)=\emptyset\\
0 & \textnormal{if}\quad A\cup D\subset B
\end{array}\right.
\label{eq:lemma22*}
\end{equation}}
\label{lemma22*}
\end{lem}

\begin{rem}
If $A=\emptyset$ and considering that $|2^\emptyset|=1$, lemma \ref{lemma22*} reduces to l
\end{rem}

\begin{proof}
For all $(C,D), (A,B)\in\mathcal Q$ with $(C,D)\subseteq(A,B)$,
\[
\sum_{\substack{(X,Y)\in\mathcal Q \\ (C,D)\subseteq (X,Y)\subseteq (A,B)}}
\left(-1\right)^{|Y|}=
\sum_{C\subseteq X\subseteq A}
\sum_{\left(X\cup D\right)\subseteq Y\subseteq B}
\left(-1\right)^{|Y|}=\textnormal{lemma \ref{lemma22}}=
\sum_{\substack{C\subseteq X\subseteq A \\ X\cup D =B }}
\left(-1\right)^{|B|}=
\]
\[
=
\left\{
\begin{array}{cl}
\left(-1\right)^{|B|}|2^{\left(A\cap D\right)\setminus C}|
& \textnormal{if}\quad A\cup D=B\\
0 & \textnormal{if}\quad A\cup D\subset B.
\end{array}\right.
\]
\end{proof}

\begin{lem}
For all $(A,B)\in\mathcal Q$ 
\small{
\begin{equation}
\sum_{
\substack{
(C,D)\in\mathcal Q \\ 
(C,D)\subseteq (A,B)
}
}
\left(-1\right)^{|D|+|C|}=
\sum_{
\substack{
(C,D)\in\mathcal Q \\ 
(C,D)\subseteq (A,B)
}
}
\left(-1\right)^{|D\setminus C|}
=
\left\{
\begin{array}{ll}
0 
& \textnormal{if}\quad A\neq B \ \textnormal{i.e.}\ B\setminus A \neq \emptyset 
\\
1 & \textnormal{if}\quad A= B \ \textnormal{i.e.}\ B\setminus A =\emptyset.
\end{array}\right.
\label{eq:lemma21+}
\end{equation}
}
\label{lemma21+}
\end{lem}

\begin{rem}
Note that if $A=\emptyset$ lemma \ref{lemma21+} states that for all $(\emptyset ,B)\in\mathcal Q$ 
\begin{equation}
\sum_{
(\emptyset ,D)\subseteq (\emptyset,B)
}
\left(-1\right)^{|D|}=
\left\{
\begin{array}{ll}
0 
& \textnormal{if}\quad B\setminus\emptyset \neq \emptyset 
\\
1 & \textnormal{if}\quad A= B =\emptyset.
\end{array}\right.
\end{equation}
that is lemma \ref{lemma21}.
\end{rem}

\begin{proof}
For all $(A,B)\in\mathcal Q$, 
\small{
\begin{equation}
\sum_{
\substack{
(C,D)\in\mathcal Q \\ 
(C,D)\subseteq (A,B)
}
}
\left(-1\right)^{|D|+|C|}=
\sum_{C\subseteq A}
\left(-1\right)^{|C|}
\sum_{C\subseteq D\subseteq B}
\left(-1\right)^{|D|}=\textnormal{lemma \ref{lemma22}}
=
\left\{
\begin{array}{ll}
0 
& \textnormal{if}\quad C\subseteq A\subset B 
\\
1 & \textnormal{if}\quad A= B.
\end{array}\right.
\end{equation}
}
Note that if $A=B$
\[
\sum_{C\subseteq B}
\left(-1\right)^{|C|}
\sum_{C\subseteq D\subseteq B}
\left(-1\right)^{|D|}=
\left(-1\right)^{|B|}
\left(-1\right)^{|B|}=1.
\]
\end{proof}

\begin{lem}
Suppose that $(C,D), (A,B)\in\mathcal Q$ with $(C,D)\subseteq(A,B)$, then
\small{\begin{equation}
\sum_{\substack{(X,Y)\in\mathcal Q \\ (C,D)\subseteq (X,Y)\subseteq (A,B)}}
\left(-1\right)^{|X|+|Y|}=
\sum_{\substack{(X,Y)\in\mathcal Q \\ (C,D)\subseteq (X,Y)\subseteq (A,B)}}
\left(-1\right)^{|Y\setminus X|}=
\left\{
\begin{array}{ll}
\left(-1\right)^{|B\setminus A|}=	\left(-1\right)^{|D\setminus C|}
& \textnormal{if}\quad B\setminus A=D\setminus C\\
0 & \textnormal{otherwise.}
\end{array}\right.
\label{eq:lemma22+}
\end{equation}}
\label{lemma22+}
\end{lem}

\begin{proof}
Let us suppose that $(C,D), (A,B)\in\mathcal Q$ with $(C,D)\subseteq(A,B)$, thus 

\[
\sum_{
\substack{
(X,Y)\in\mathcal Q \\ 
(C,D)\subseteq (X,Y)\subseteq (A,B)
}
}
\left(-1\right)^{|X|+|Y|}=
\sum_{C\subseteq X\subseteq A}
\left(-1\right)^{|X|}
\sum_{D\cup X \subseteq Y\subseteq B}
\left(-1\right)^{|Y|}=\textnormal{lemma \ref{lemma22}}
=
\]
\[
=
\left\{
\begin{array}{cl}
0 & \textnormal{if}\quad A\cup D\subset B \left(\textnormal{ and then }D\cup X\subset B \textnormal{ for all } X\subseteq A\right)
\\
\left(-1\right)^{|B|}\underset{\substack{C\subseteq X\subseteq A \\D\cup X=B }}{\sum}\left(-1\right)^{|X|} & \textnormal{if}\quad A\cup D=B.
\end{array}\right.
\]
Now we further examine the case $A\cup D=B$. 
\[
\sum_{
\substack{
(X,Y)\in\mathcal Q \\ 
(C,D)\subseteq (X,Y)\subseteq (A,B)
}
}
\left(-1\right)^{|X|+|Y|}=
\left(-1\right)^{|B|}\underset{\substack{C\subseteq X\subseteq A \\D\cup X=B }}{\sum}\left(-1\right)^{|X|} 
=
\left(-1\right)^{|B|}\sum_{X'\subseteq (A\cap D)\setminus C}\left(-1\right)^{|C\cup (A\setminus D)|}\left(-1\right)^{|X'|}=
\]
\[
=\left(-1\right)^{|B| + |C| + |A\setminus D|}
\sum_{X'\subseteq (A\cap D)\setminus C}\left(-1\right)^{|X'|}
=(\textnormal{lemma \ref{lemma21})}
=
\left\{
\begin{array}{cl}
\left(-1\right)^{|B| + |C| + |A\setminus D|}& \textnormal{if}\quad C=A\cap D
\\
0 & \textnormal{if}\quad C\neq A\cap D.
\end{array}\right.
\]
Thus we have proved that 
\[
\sum_{
\substack{
(X,Y)\in\mathcal Q \\ 
(C,D)\subseteq (X,Y)\subseteq (A,B)
}
}
\left(-1\right)^{|X|+|Y|}=
\left\{
\begin{array}{cl}
\left(-1\right)^{|B| + |C| + |A\setminus D|}=
\left(-1\right)^{|B\setminus A|}=
\left(-1\right)^{|D\setminus C|}
& \textnormal{if}\quad D\cup A=B \textnormal{ and } D\cap A =C
\\
0 & \textnormal{otherwise.}
\end{array}\right.
\]
To complete the proof we show that 
$B\setminus A=D\setminus C$ iff ($A\cap B=C$ and $A\cup D=B$). 
Indeed if ($A\cap B=C$ and $A\cup D=B$) thus 
$B\setminus A= (D\cup A)\setminus A=D\setminus A= D\setminus (D\cap A)= D\setminus C$. 
Now suppose that $B\setminus A=D\setminus C$. 
If $D\cup A\neq B$, it exists $x*\in B\setminus (A\cup D)$ then $x*\in B\setminus A$ and $x*\notin D\setminus C$ and we get the contradiction that $B\setminus A \neq D\setminus C$. 
If $A\cap D\neq C$ it exists $y*\in (A\cap D)\setminus C$ and in this case 
$y*\in D\setminus C$ and $y*\notin B\setminus A$ contradicting the hypothesis that $B\setminus A=D\setminus C$.
\end{proof}

\begin{proof} of proposition \ref{pro:mobius}.\\
\eqref{eq:mobius1} $\rightarrow$ \eqref{eq:mobius2}. 
For all $(A,B)\in\mathcal Q$,
\[
\sum_{\emptyset\subseteq X\subseteq A}
\left[
\left(-1\right)^{\abs{X}}{
\sum_{\substack{(C,D)\in\mathcal Q \\(C,D)\subseteq (A\setminus X,B\setminus X)}}\left(\left(-1\right)^{\abs{B\setminus A}-\abs{D\setminus C}}f(C,D)\right)
} \right]=
\]
\[
=\left(-1\right)^{\abs{B\setminus A}}
\sum_{\emptyset\subseteq X\subseteq A}
\left[
\left(-1\right)^{\abs{X}}
\sum_{\substack{(C,D)\in\mathcal Q \\(C,D)\subseteq (A\setminus X,B\setminus X)}}\left(\left(-1\right)^{\abs{D\setminus C}}f(C,D)\right)
\right]= \  \eqref{eq:mobius1}
\]
\[
=
\left(-1\right)^{\abs{B\setminus A}}
\sum_{\emptyset\subseteq X\subseteq A}
\left[
\left(-1\right)^{\abs{X}}
\sum_{\substack{(C,D)\in\mathcal Q \\(C,D)\subseteq (A\setminus X,B\setminus X)}}\left(\left(-1\right)^{\abs{D\setminus C}}\sum_{\substack{(T,Z)\in\mathcal Q \\(T,Z)\subseteq (C,D)}}g(T,Z)\right)
\right]=\ \textnormal{first inversion}
\]
\[
=
\left(-1\right)^{\abs{B\setminus A}}
\sum_{\emptyset\subseteq X\subseteq A}
\left[\left(-1\right)^{\abs{X}}
\sum_{\substack{(C,D)\in\mathcal Q \\(C,D)\subseteq (A\setminus X,B\setminus X)}}
\left(g(C,D)
\sum_{\substack{(T,Z)\in\mathcal Q \\(C,D)\subseteq (T,Z)\subseteq (A\setminus X,B\setminus X)}}
\left(-1\right)^{\abs{Z\setminus T}}\right)\right]
=\ \textnormal{lemma \ref{lemma22+}}
\]
\[=
\left(-1\right)^{\abs{B\setminus A}}
\sum_{\emptyset\subseteq X\subseteq A}
\left[
\left(-1\right)^{\abs{X}}
\sum_{
\substack{
(C,D)\in\mathcal Q 
\\(C,D)\subseteq (A\setminus X,B\setminus X)
\\D\setminus C= (B\setminus X)\setminus (A\setminus X)=B\setminus A}
}
\left(g(C,D)\left(-1\right)^{|B\setminus A|}\right)
\right]=
\]
\[
=
\sum_{\emptyset\subseteq X\subseteq A}
\left[\left(-1\right)^{\abs{X}}
\sum_{
\substack{
(C,D)\in\mathcal Q 
\\(C,D)\subseteq (A\setminus X,B\setminus X)
\\D\setminus C= (B\setminus X)\setminus (A\setminus X)=B\setminus A}
}
g(C,D)\right]
=\sum_{\emptyset\subseteq X\subseteq A}
\left[\left(-1\right)^{\abs{X}}
\sum_{
\substack{
(C,D)\in\mathcal Q 
\\(C,D)\subseteq (A\setminus X,B\setminus X)
\\D\cap (A\setminus X)=C
\\D\cup (A\setminus X) =B\setminus X}
}
g(C,D)\right]\]
\[
=
\sum_{\emptyset\subseteq X\subseteq A}
\left[
\left(-1\right)^{\abs{X}}
\sum_{
\substack{
(C,D)\in\mathcal Q 
\\(C,D)\subseteq (A\setminus X,B\setminus X)
\\D\setminus C= B\setminus A}
}
g(C,D)
\right]=
\sum_{\emptyset\subseteq X\subseteq A}
\left[
\left(-1\right)^{\abs{X}}
\sum_{
C\subseteq A\setminus X}
g(C,C\cup (B\setminus A))
\right]=
\]
\[
=\ \textnormal{(second inversion)}\ =
\sum_{\emptyset\subseteq X\subseteq A}
\left[
g(X,X\cup (B\setminus A))
\sum_{
Y\subseteq A\setminus X}
\left(-1\right)^{\abs{Y}}
\right]=
\]
\[
\ \textnormal{(\ref{lemma21})}\ 
=
g(A,A\cup (B\setminus A))=g(A,B).
\]
\eqref{eq:mobius2} $\rightarrow$ \eqref{eq:mobius1}. 
For all $(A,B)\in\mathcal Q$, 
\[
\sum_{\substack{(C,D)\in\mathcal Q\\(C,D)\subseteq (A,B)}}g(C,D)=
\sum_{\substack{(C,D)\in\mathcal Q\\(C,D)\subseteq (A,B)}}
\left[
\sum_{\emptyset\subseteq X\subseteq C}
\left(-1\right)^{|X|}g^*(C\setminus X, D\setminus X)
\right]=
\]
\[
=\sum_{\substack{(C,D)\in\mathcal Q\\(C,D)\subseteq (A,B)}}
\left[
\sum_{\emptyset\subseteq X\subseteq C}
\left(
\left(-1\right)^{|X|}
\sum_{\substack{(T,Z)\in\mathcal Q\\(T,Z)\subseteq (C\setminus X, D\setminus X)}}
\left(-1\right)^{|(D\setminus X)\setminus (C\setminus X)|-|Z\setminus T|}f(T,Z)
\right)
\right]=
\]
\[
=
\sum_{\substack{(C,D)\in\mathcal Q\\(C,D)\subseteq (A,B)}}
\left[
\sum_{\emptyset\subseteq X\subseteq C}
\left(
\left(-1\right)^{|X|}
\sum_{\substack{(T,Z)\in\mathcal Q\\(T,Z)\subseteq (C\setminus X, D\setminus X)}}
\left(-1\right)^{|D\setminus C|-|Z\setminus T|}f(T,Z)
\right)
\right]
=
\]
\[
=
\sum_{\substack{(C,D)\in\mathcal Q\\(C,D)\subseteq (A,B)}}
\left[
\left(-1\right)^{|D\setminus C|}
\sum_{\emptyset\subseteq X\subseteq C}
\left(
\left(-1\right)^{|X|}
\sum_{\substack{(T,Z)\in\mathcal Q\\(T,Z)\subseteq (C\setminus X, D\setminus X)}}
\left(-1\right)^{|Z\setminus T|}f(T,Z)
\right)
\right]
=\ \textnormal{(first inversion)}
\]
\[
=
\sum_{\substack{(C,D)\in\mathcal Q\\(C,D)\subseteq (A,B)}}
\left[
\left(-1\right)^{|D\setminus C|}
\sum_{\substack{(T,Z)\in\mathcal Q\\(T,Z)\subseteq (C, D)}}
\left(
\left(-1\right)^{|Z\setminus T|}f(T,Z)
\sum_{\emptyset\subseteq X\subseteq C\setminus Z}
\left(-1\right)^{|X|}
\right)
\right]
=\ \textnormal{(lemma \ref{lemma21})}
\]
\[
=
\sum_{\substack{(C,D)\in\mathcal Q\\(C,D)\subseteq (A,B)}}
\left[
\left(-1\right)^{|D\setminus C|}
\sum_{\substack{(T,Z)\in\mathcal Q\\(T,Z)\subseteq (C, D)\\C\setminus Z=\emptyset}}
\left(-1\right)^{|Z\setminus T|}f(T,Z)
\right]
=
\]
\[
=
\sum_{\substack{(C,D)\in\mathcal Q\\(C,D)\subseteq (A,B)}}
\left[
\left(-1\right)^{|D\setminus C|}
\sum_{\emptyset\subseteq X\subseteq C}
\left(
\sum_{C\subseteq Y\subseteq D}
\left(-1\right)^{|Y\setminus X|}f(X,Y)
\right)
\right]
=
\]
\[
=
\sum_{\substack{(C,D)\in\mathcal Q\\(C,D)\subseteq (A,B)}}
\left[
\left(-1\right)^{|D\setminus C|}f(C,D)
\sum_{C\subseteq X\subseteq D\cap A}
\left(
\sum_{D\subseteq Y\subseteq B}
\left(-1\right)^{|Y\setminus X|}
\right)
\right]
=\ \textnormal{(being $X\subseteq D\cap A\subseteq D\subseteq Y$)}
\]
\[
=
\sum_{\substack{(C,D)\in\mathcal Q\\(C,D)\subseteq (A,B)}}
\left[
\left(-1\right)^{|D\setminus C|}f(C,D)
\sum_{C\subseteq X\subseteq D\cap A}
\left(
\left(-1\right)^{|X|}
\sum_{D\subseteq Y\subseteq B}
\left(-1\right)^{|Y|}
\right)
\right]
=\ \textnormal{(lemma \ref{lemma21})}
\]
\[
=
\sum_{\substack{(C,B)\in\mathcal Q\\(C,B)\subseteq (A,B)}}
\left[
\left(-1\right)^{|B\setminus C|}f(C,B)
\sum_{C\subseteq X\subseteq A}
\left(-1\right)^{|X|}
\left(-1\right)^{|B|}
\right]
= 
\]
\[
=\left[\left(-1\right)^{|B|}
\right]^2
\sum_{C\subseteq A}
\left[
\left(-1\right)^{|C|}f(C,B)
\sum_{C\subseteq X\subseteq A}
\left(-1\right)^{|X|}
\right]
=\ \textnormal{(lemma \ref{lemma21})}\ 
=
\left(-1\right)^{|A|}f(A,B)\left(-1\right)^{|A|}=f(A,B).
\]
\end{proof}

\begin{proof} of proposition \ref{pro:mobius1}.\\
1) and 2) follow directly by the conditions 
$$\mu_r(\emptyset,\emptyset)=0,\quad \mu_r(N,N)=1,\quad \textnormal{and} 
\quad \mu_r(A,B)=\underset{{(C,D)\subseteq (A,B)}}{\sum}m(C,D).$$
To prove 3) and 4) it is sufficient to note that for any function $f:\mathcal Q\rightarrow\rea$ and for all $(A,B),(C,D)\in\mathcal Q$, the monotonicity condition
\begin{equation}
f(C,D)\le f(A,B)\quad\textnormal{whenever}\quad (C,D)\subseteq(A,B) 
\label{eq:mon}
\end{equation} 
is equivalent to the following two statements
\begin{equation}
f(A\setminus \{a\},B)\le f(A,B)\quad\textnormal{for all}\quad a\in A 
\label{eq:i}
\end{equation}
and 
\begin{equation}
f(A\setminus \{b\},B\setminus \{b\})\le f(A,B)\quad\textnormal{for all}\quad b\in B. 
\label{eq:ii}
\end{equation}
\eqref{eq:mon} trivially imply \eqref{eq:i} and \eqref{eq:ii}.
Suppose that $(C,D)\subseteq(A,B)$ and note that $C\subseteq A\cap D$. 
By using respectively \eqref{eq:i} and \eqref{eq:ii}, we get:
\[
f(C,D)\le f(A\cap D, D)= f(A\setminus (B\setminus D), B\setminus (B\setminus D))\le f(A,B).
\]
\end{proof}

\end{document}